\documentclass[11pt]{amsart}
\usepackage{amssymb,amscd}
\usepackage{pb-diagram}
\usepackage{verbatim}
\usepackage{rotating}
\usepackage{graphicx}
\usepackage{amssymb,amscd,enumerate}
\newtheorem*{Whitney towers}{Theorem~\ref{Whitney towers}}
\newtheorem*{h-towers}{Theorems ~\ref{half} \& \ref{$(n)$-solvable}}

\newtheorem*{surgery curves}{Theorem~\ref{surgery curves}}
\newtheorem*{cg=0}{Theorem~\ref{vanish}}

\newtheorem{thm}{Theorem}[section]
\newtheorem{lem}[thm]{Lemma}
\newtheorem{cor}[thm]{Corollary}

\newtheorem{prop}[thm]{Proposition}
\newtheorem{defn}[thm]{Definition}

\newtheorem{ex}[thm]{Example}

\numberwithin{equation}{section}
\numberwithin{figure}{section}

\newcommand{\Z}{\mathbb{Z}}

\newcommand{\Q}{\mathbb{Q}}

\newcommand{\G}{\Gamma}

\newcommand{\ra}{\longrightarrow}

\def\ov{\overline}

\def\s{\sigma}
\def\lra{\longrightarrow}

\title{Derivatives of Knots and Second-Order Signatures}
\author{Tim D. Cochran$^{\dag}$}
\address{Department of Mathematics MS-136, Rice University, PO 1892, Houston, Texas, 77251-1892}
\email{cochran@rice.edu}

\author{Shelly Harvey$^{\dag\dag}$}
\address{Department of Mathematics MS-136, Rice University, PO 1892, Houston, Texas, 77251-1892}
\email{shelly@rice.edu}

\author{Constance Leidy $^{\dag\dag\dag}$}
\address{Wesleyan University, Wesleyan Station, Middletown, CT 06459}
\email{cleidy@wesleyan.edu}

\thanks{\noindent $^{\dag}$Partially supported by the National Science Foundation NSF-DMS-0706929}
\thanks{ $^{\dag\dag}$Partially supported
by NSF-DMS-0539044, NSF-CAREER-DMS-0748458, and The Alfred P. Sloan Foundation}
\thanks{ $^{\dag\dag\dag}$Partially supported
by NSF-DMS-0805867 }

\begin{document}

\begin{abstract} We define a set of ``second-order'' $L^{(2)}$-signature invariants for any algebraically slice knot. These obstruct a knot's being a slice knot and generalize Casson-Gordon invariants, which we consider to be ``first-order signatures''. As one application we prove: If $K$ is a genus one slice knot then, on any genus one Seifert surface $\Sigma$, there exists a homologically essential simple closed curve $J$ of self-linking zero, which has vanishing zero-th order signature and a vanishing first-order signature. This extends theorems of Cooper and Gilmer.
We introduce a geometric notion, that of a derivative of a knot with respect to a metabolizer. We also introduce a new equivalence relation, generalizing homology cobordism, called null-bordism.
\end{abstract}

\maketitle
\section{Introduction}\label{sec:Introduction}

A \textbf{knot} $K$ is the image of a tame embedding of an oriented circle in $S^3$. A \textbf{slice knot} is a knot that bounds an embedding of a $2$-disk in $B^4$. We wish to consider both the \emph{smooth} category and the \emph{topological} category (in the latter case all embeddings are required to be flat). The question of which knots are slice knots was first considered by Kervaire and Milnor in the early 1960's in their study of isolated singularities of $2$-spheres in $4$-manifolds. The question of which knots are slice knots lies at the heart of the topological classification of $4$-dimensional manifolds. Moreover the question of which knots are topologically slice but not smoothly slice may be viewed as ``atomic'' for the question of which topological $4$-manifolds admit distinct smooth structures.

In the 1960's Murasugi, Tristram, Levine and Milnor defined what we shall call \emph{classical} or \emph{abelian} or \emph{zero-th order signatures} for a knot $K$ and showed that these obstruct a knot's being a slice knot. In the 1970's Casson-Gordon defined a set of numbers, that were the first of what we shall call \emph{metabelian} or \emph{first-order signatures}. They showed that these also obstruct a knot's being a slice knot and moreover are independent of the zero-th order signatures ~\cite{CG1}\cite{CG2}. Later, Gilmer showed that, for genus one knots, the Casson-Gordon invariants could be estimated in terms of classical signatures of the components of certain knots lying on a Seifert surface for $K$ ~\cite[Thm. 3]{Gi2}  In particular, (unpublished) improvements of Gilmer's work by D. Cooper yield  the following (a proof of this result can be found in ~\cite[Thm. 5.2]{COT2}).

\begin{thm}[Cooper (see also {\cite{Gi3}\cite[Thm. 4]{Gi2}\cite[Thm. 5.2]{COT2}})]\label{thm:CooperThm} If $K$ is a genus one slice knot that does not have Alexander polynomial $1$ then, on any genus one Seifert surface $\Sigma$, there exist precisely two homologically essential simple closed curves of self-linking zero, one of which, when viewed as a knot in $S^3$, has average classical signature zero.
\end{thm}

In fact it is well known that if a knot $K$ admits a genus one Seifert surface on which lies a homologically essential simple closed curve of self-linking zero that is a slice knot, then $K$ itself is a slice knot.

\textbf{Conjecture}: The converse of the above is true, that is, a genus one knot is smoothly slice if and only if there exists a homologically essential simple closed curve of self-linking zero on $\Sigma$ that is itself a smoothly slice knot.

We extend the theorem of Cooper and give further evidence for the veracity of this conjecture.

\newtheorem*{thm:main}{Theorem~\ref{thm:main}}
\begin{thm:main} If $K$ is a genus one slice knot then, on any genus one Seifert surface, there exists a homologically essential simple closed curve of self-linking zero, that has vanishing zero-th order signature and a vanishing first-order signature. (Beware that if $\Delta_K(t)=1$ then these  signatures are zero by our definition).
\end{thm:main}

More generally, Gilmer's work suggests, in our language, that the first-order signatures of a knot are related to the zero-th order signatures of certain knots lying on a Seifert surface for $K$. In this paper we extend this work by first defining \emph{second-order signatures} for any algebraically slice knot $K$. These are defined in terms of the first-order signatures of an associated link lying on a Seifert surface for $K$, that we call a \emph{derivative} of $K$ (with respect to a metabolizer of its Seifert form). Briefly, given a metabolizer, $\mathfrak{m}$, for the Seifert form of $K$, $\frac{\partial K}{\partial \mathfrak{m} }$ is a link, embedded on $\Sigma$, representing a basis for $\mathfrak{m}$. We then show, under certain circumstances, that these signatures obstruct $K$ being a slice knot.

We give examples that show that these obstructions are stronger than those imposed by any abelian or metabelian invariants. Higher-order signatures that are (in some cases) stronger than
classical and Casson-Gordon invariants were first described in work of Cochran-Orr-Teichner ~\cite[Theorem 4.6]{COT}, and also used in \cite{COT2}\cite{Ki2}\cite{Hor3}. The present paper improves on these (at the level of second-order only) in several aspects. Firstly, these previous signatures vanish for genus one knots.  Secondly, these previous signatures were defined not so much as invariants but rather as obstructions. Thirdly, and most importantly, the previous signatures were parameterized by very large sets. Since the theorems were of the nature: ``If $K$ is slice then \emph{one} of the second-order signatures vanishes'', the vagueness and infinitude of the index set of such signatures makes them often useless. More recently, P. Horn gave a new definition of higher-order signatures in \cite{Hor3} that overcame the first and second objections above, but not the third. The authors introduced other techniques in ~\cite{CHL1}\cite{CHL1A}\cite{CHL3}, giving higher-order signature obstructions that are non-zero even for genus one knots. But even these signatures were not defined as invariants for all knots and moreover the technique seemed applicable only to a very special class of knots obtained from ribbon knots by iterated satellite constructions. The present paper eliminates, to some extent, all of the above limitations (especially for genus one knots). We show that the index set of our second-order signatures is often finite. Thus it is possible, at least in theory, to check \emph{all} of them.

If $K$ has genus greater than $1$ then its derivatives are links. In this case much less has been developed in the literature. Here first-order signatures of $K$ are related to zero-th order signatures and the Alexander nullities of the derivative links (for Casson-Gordon invariants this was seen in ~\cite{Gi3} although Gilmer informs us that the proofs of Corollaries 0.2 and 0.3 of this work are now known to contain gaps). At the first-order level we are able to recover an $L^{(2)}$-version of Gilmer's aforementioned results and extend Cooper's theorem (which is the case $c=1$)to knots of higher genera:

\newtheorem*{thm:cor}{Corollary~\ref{cor:sliceimpliesrhobound}}
\begin{thm:cor} If $K$ is a slice knot, $P$ is a Lagrangian corresponding to a slice disk, $J=\frac{\partial K}{\partial \mathfrak{m} }$ is a $c$-component link where $\mathfrak{m}$ is metabolizer that represents $P$, and $f:\pi_1(M_J)\to A\cong\Z^k$ is an epimorphism, then
$$
|\rho^f_0(J)|\leq c-1-\eta(J,f),
$$
where $\rho^f_0(J)$ is the zero-th order signature of the derivative link and $\eta(J,f)$ is its Alexander nullity.
\end{thm:cor}

If $K$ has genus greater than $1$ then its \emph{second-order} signatures are related to first-order signatures of its derivative links as well as to certain first-order nullities that we will not treat here in full generality. We state a result for higher-genus knots only in cases where the derivative links have maximal values of these nullities (for definitions see Section~\ref{sec:notation}).

\newtheorem*{thm:main2}{Theorem~\ref{thm:main2}}
\begin{thm:main2} Suppose $K$ is a slice knot with the property that: for each Lagrangian $P$ for which the first-order signature of $K$ corresponding to $P$ vanishes it is possible to choose the corresponding derivative to be a link of maximal Alexander nullity. Then some member of a complete set of second-order signatures  has absolute value at most genus$(\Sigma)-1$.  Moreover, if it is possible to choose each derivative to be an infected trivial link then any complete set of second-order signatures contains zero.
\end{thm:main2}

Examples are given of higher-genus non-satellite knots with vanishing classical and metabelian invariants for which the second-order signatures of Theorem~\ref{thm:main2} obstruct their being slice knots. Moreover these examples, like those of ~\cite{CHL3}, cannot be detected by the techniques of ~\cite{COT} (they are even distinct up to concordance from those considered there). Since they are not formed by iterated satellite constructions, the techniques of ~\cite{CHL3} cannot, in an obvious way, be applied.

All of the results of this paper hold in more general settings. For example, rather than merely being obstructions to a knots being a slice knot, the second-order signatures obstruct a knots being $(2.5)$-solvable. Here we refer to the $(n)$-solvable filtration, $\{\mathcal{F}_{(n)}\}$, of the knot concordance group due to Cochran-Orr-Teichner ~\cite[Section 7,8]{COT}. Here, for simplicity we suppress this level of generality and take the philosophy that we are generalizing the seminal work of Gilmer, Cooper and Casson-Gordon from the 1970's and 1980's. However, in Section~\ref{sec:nsolvable}, we state and prove some of these generalizations.

\section{Notation and Background}\label{sec:notation}

If $K$ is a knot or link in $S^3$, let $S^3-K$ denote the \textbf{exterior} of an open tubular neighborhood of $K$ in $S^3$. Similarly if $\Delta\hookrightarrow B^4$ is a slice disk, let $B^4-\Delta$ denote the exterior of an open tubular neighborhood of $\Delta$ in $B^4$.Let $M_K$ denote the closed $3$-manifold obtained by zero-framed Dehn surgery on the components of $K$. If $K$ is a knot we let $\mathcal{A}_0(K)$ denote the \textbf{rational Alexander module} of $K$
$$
\mathcal{A}_0(K)\equiv H_1(S^3-K;\Q[t,t^{-1}])\cong H_1(M_K;\Q[t,t^{-1}]),
$$
where the latter follows since the longitude of a knot is trivial in the Alexander module. The rank of $\mathcal{A}_0(K)$ as a rational vector space is the degree of the Alexander polynomial of $K$, $\Delta_K(t)$. If the degree of $\Delta_{K}(t)$ is $2d$, it is well known that $d$ is at most the genus, $g$, of any Seifert surface for $K$. In addition, there is a nonsingular \textbf{classical Blanchfield linking form}, $\mathcal{B}\ell_0$ defined on $\mathcal{A}_0(K)$. A submodule $P\subset \mathcal{A}_0(K)$ is a \textbf{Lagrangian} if $P=P^\perp$ with respect to $\mathcal{B}\ell_0$ on $\mathcal{A}_0(K)$. Here
$$
P^\perp=\{x\in \mathcal{A}_0(K)|~\mathcal{B}\ell_0(x,p)=0~ \text{for every } p\in P\}.
$$
It follows from the nonsingularity of $\mathcal{B}\ell_0$ that any Lagrangian is a $d$-dimensional vector subspace (half that of $\mathcal{A}_0(K)$). A submodule $P\subset \mathcal{A}_0(K)$ is \textbf{isotropic} if $P\subset P^\perp$ with respect to $\mathcal{B}\ell_0$, that is, for all $x,y\in P$, $\mathcal{B}\ell_0(x,y)=0$.

More generally, we extend the notion of Alexander module and Blanchfield form to pairs $(J,f)$ where $J=\{J_1,...,J_m\}$ is an ordered, oriented link with trivial linking numbers and $f:\pi_1(M_J)\to A\cong\Z^k$ is an epimorphism. A specific identification of $A$ with $\Z^k$ is \emph{not} assumed. The \textbf{rational Alexander module} of $(J,f)$, denoted $\mathcal{A}_0^f(J)$, is the $\mathbb{Q}[\mathbb{Z}^k]$-torsion submodule of $H_1(M_J;\mathbb{Q}[\mathbb{Z}^k])$. Throughout we will abuse  notation by writing $H_1(M_J;\mathbb{Q}[\mathbb{Z}^k])$ instead of the more proper $H_1(M_J;\mathbb{Q}[A])$. Note that if $f$ is the zero map then $\mathcal{A}_0^f(J)=0$. There is also nonsingular \textbf{classical Blanchfield linking form} defined on $\mathcal{A}_0^f(J)$,
$$
\mathcal{B}\ell_0^f(J):\mathcal{A}_0^f(J) \to (\mathcal{A}_0^f(J))^{\#}\equiv \overline{\text{Hom}_{\mathcal{R}}(\mathcal{A}_0^f(J), \Q(x_1,...,x_k)/\mathcal{R})}.
$$
where $\mathcal{R}\equiv \mathbb{Q}[A]$ ~\cite[Theorem 2.3]{Lei3}, given by the composition of Poincar\'{e} duality, the inverse of a Bockstein and a Kronecker map:
$$
H_1(M_J;\mathcal{R})\overset{PD}{\to}H^2(M_J;\mathcal{R})\overset{B^{-1}}{\rightarrow}
H^1(M_J;\Q(x_1,...,x_k)/\mathcal{R})\overset{\kappa}{\to}\text{Hom}_{\mathcal{R}}(\mathcal{A}_0^f(J), \Q(x_1,...,x_k)/\mathcal{R}).
$$
This reduces, in the case that $J$ is a knot, to the former definition.
The  \textbf{Alexander nullity}, $\boldsymbol{\eta(J,f)}$, of $(J,f)$ is the $\mathcal{R}$-rank of $H_1(M_J;\mathbb{Q}[A])$. We say that $(J,f)$ has \textbf{maximal Alexander nullity} if $f$ is non-trivial and $\eta(J,f)=m-1$. If $J$ is a knot and $f$ is non-trivial then $\eta(J,f)=0$ since the classical Alexander module is a torsion module. More generally, if $f$ is the abelianization map then any link with trivial Milnor's invariants (such as a  boundary link) has maximal Alexander nullity ~\cite{Hi}.

If $G$ is a group then the terms of the \textbf{derived series} of $G$ are defined by $G^{(0)}\equiv G$ and $G^{(n+1)}\equiv [G^{(n)},G^{(n)}]$. The terms of the \textbf{rational derived series} of $G$ are defined by $G_r^{(0)}\equiv G$  and
$$
G_r^{(n+1)}\equiv \{x\in G^{(n)}_r|~ x^k\in [G_r^{(n)},G_r^{(n)}] ~~\text{for some non-zero integer} ~k\}.
$$
Thus $G^{(n)}\subset G^{(n)}_r$ and they agree when $G$ is a knot group ~\cite[Section 3]{Ha1}\cite{Str}.
A group $\G$ is \textbf{poly-torsion-free-abelian} (henceforth \textbf{PTFA}) if it admits a normal
series $\{1\} = \G_0\lhd ~ \G_1\dots \lhd~ \G_n =\G$ such that each of the quotients $\G_{i+1}/\G_i$ is torsion-free abelian. Then one checks that $G/G^{(n)}_r$ is PTFA for any $n$ and $G$ ~\cite[Section 3]{Ha1}.

We describe a generalized satellite construction that has been useful in the literature. Let $R$ be a link in $S^3$ and $\{\eta_1,\eta_2,\ldots,\eta_m\}$ be an oriented trivial link in $S^3$ which misses $R$ bounding a collection of disks that meet $R$ transversely. Suppose $(K_1,K_2,\ldots,K_m)$ is an $m$-tuple of auxiliary knots. Let $R(\eta_1,\ldots,\eta_m,K_1,\ldots,K_m)$ denote the result of the operation pictured in Figure~\ref{fig:infection}. That is, for each $\eta_i$, take the embedded disk in $S^3$ bounded by $\eta_i$; cut $R$ along the disk; grab the cut strands, tie them into the knot $K_i$ (with no twisting) and reglue as shown in Figure~\ref{fig:infection}.

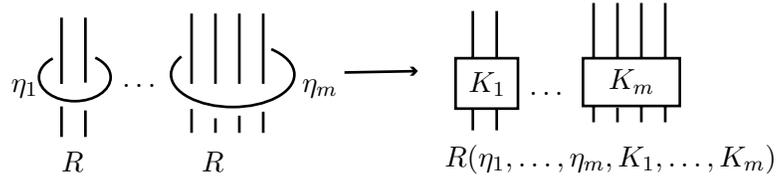
\begin{figure}[htbp]
\setlength{\unitlength}{1pt}
\begin{picture}(262,71)
\put(10,37){$\eta_1$} \put(120,37){$\eta_m$} \put(52,39){$\dots$}
\put(206,36){$\dots$} \put(183,37){$K_1$} \put(236,38){$K_m$}
\put(174,9){$R(\eta_1,\dots,\eta_m,K_1,\dots,K_m)$}
\put(29,7){$R$} \put(82,7){$R$}
\put(20,20){\includegraphics{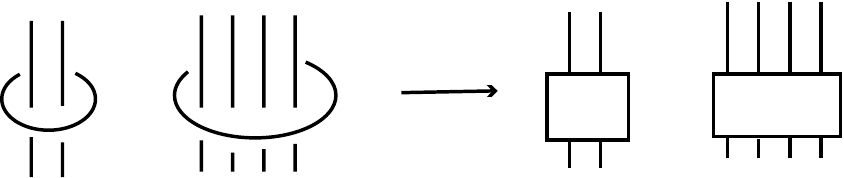}}
\end{picture}
\caption{$R(\eta_1,\dots,\eta_m,K_1,\dots,K_m)$:
Infection of $R$ by $K_i$ along $\eta_i$}\label{fig:infection}
\end{figure}
\noindent We will call this the result of \textbf{infection performed on the link $\boldsymbol{R}$ using the infection knots $\boldsymbol{K_i}$ along the curves $\boldsymbol{\eta_i}$} ~\cite{COT2}. In particular, in this paper, if we draw a band passing through a box labelled by a \emph{knot} (as for example the box labeled by $K_1$ in Figure~\ref{fig:infection}) then this means that that entire band is to be tied into that knot as shown in Figure~\ref{fig:knottedband}
for a trefoil knot.

\begin{figure}[htbp]
\setlength{\unitlength}{1pt}
\begin{picture}(165,151)
\put(0,0){\includegraphics{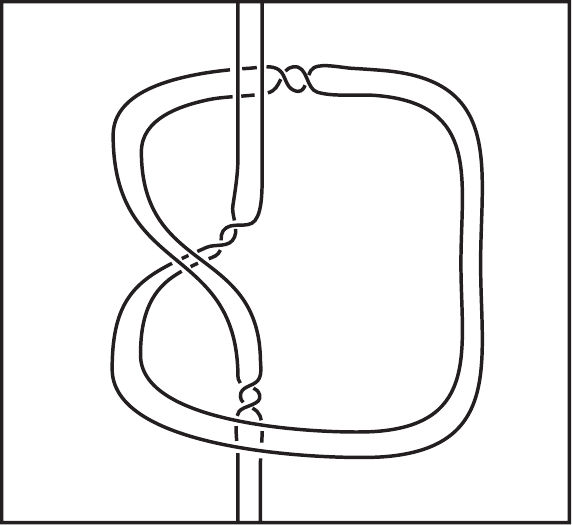}}
\end{picture}
\caption{Tying a band into a trefoil knot}\label{fig:knottedband}
\end{figure}

A link $J=\{J_1,...,J_m\}$ is an \textbf{infected trivial link} if it is obtained from the trivial link of $m$ components by a number of infections on knots along curves in the commutator subgroup.

In this paper we also need to consider \textbf{infection by a string link} as discussed in ~\cite{CO2}\cite[Section 10]{C}\cite{CFT}. We need only the following special case. A $2$-component string link is the union of two knotted arcs properly embedded in $B^2\times [0,1]$. An example of a $2$-component string link is shown on the left-hand side of Figure~\ref{fig:linkedband}. Suppose two ``bands" of a knot ($4$ strands altogether) pass through a box labeled by $L$, a $2$-component string link $L=\{L_1,L_2\}$, as indicated in the center of Figure~\ref{fig:linkedband}. In the present paper this abbreviates the following. Take untwisted parallel-push-offs of each component of $L$, thus forming a four component string link, and use this to replace the box. The final result for the example is shown on the right-hand side of Figure~\ref{fig:linkedband}. The twists at the bottom ensure that the linking numbers between the two parallels are zero. We say that the \textbf{bands of the knot are }\textbf{tied into a string link} $\boldsymbol{L}$.

\begin{figure}[htbp]
\setlength{\unitlength}{1pt}
\begin{picture}(265,221)
\put(-80,0){\includegraphics{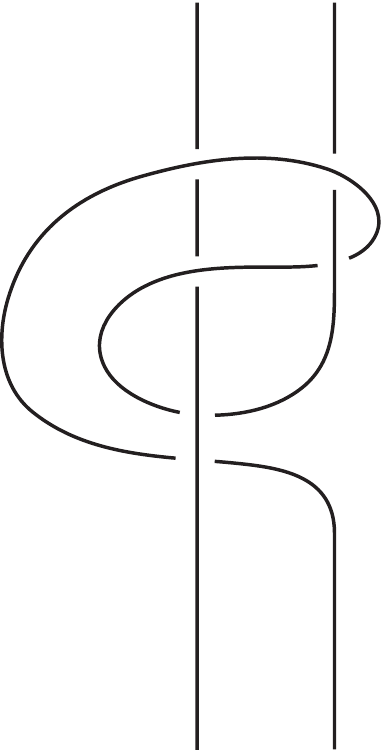}}
\put(80,0){\includegraphics{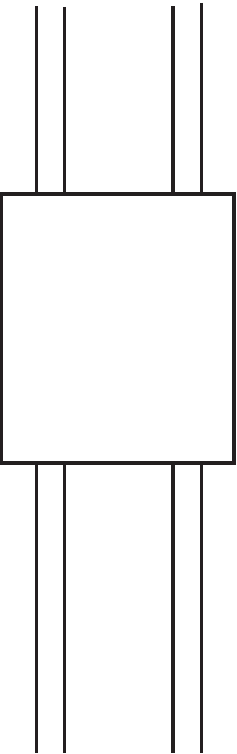}}
\put(190,0){\includegraphics{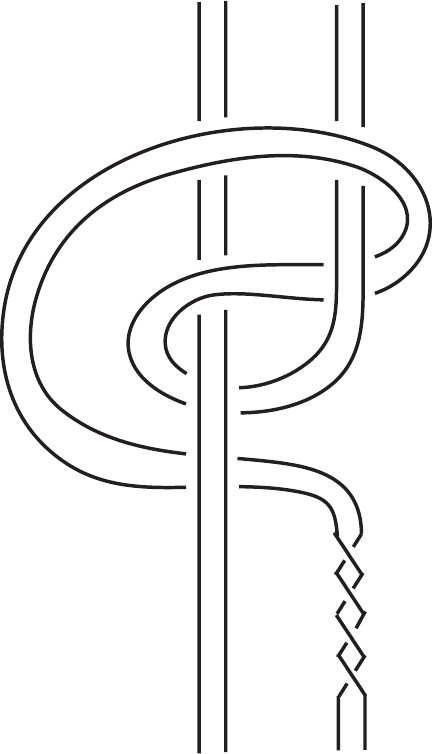}}
\put(110,120){$L$}
\end{picture}
\caption{Tying two bands into a string link}\label{fig:linkedband}
\end{figure}

The signature invariants we employ in this paper are von Neumann $\rho$-invariants. Given a compact, oriented 3-manifold $M$, a discrete group $\G$, and a representation $\phi : \pi_1(M)
\to \G$, the \textbf{von Neumann
$\boldsymbol{\rho}$-invariant} was defined by Cheeger and Gromov by choosing a Riemannian metric and using $\eta$-invariants associated to $M$ and its covering space induced by $\phi$. It can be thought of as an oriented homeomorphism invariant associated to an arbitrary regular covering space of $M$ ~\cite{ChGr1}. If $(M,\phi) = \partial
(W,\psi)$ for some compact, oriented 4-manifold $W$ and $\psi : \pi_1(W) \to \G$, then it is known that $\rho(M,\phi) =
\s^{(2)}_\G(W,\psi) - \s(W)$ where $\s^{(2)}_\G(W,\psi)$ is the
$L^{(2)}$-signature (von Neumann signature) of the equivariant intersection form defined on
$H_2(W;\mathbb{Z}\G)$ twisted by $\psi$ and $\sigma(W)$ is the ordinary
signature of $W$ ~\cite{LS}. Thus the $\rho$-invariants should be thought of as \textbf{signature defects}. They were first used to detect non-slice knots in ~\cite{COT}. For a more thorough discussion see ~\cite[Section 2]{CT}~\cite[Section 2]{COT2}. All of the coefficient systems $\G$ in this paper will be of the form $\pi/\pi^{(n)}_r$ where $\pi$ is the fundamental group of a space. Hence all such $\G$ will be PTFA as above. Aside from the definition, a few crucial properties that we use in this paper are:
\begin{itemize}
\item [(1)] If $\phi$ factors through $\phi' : \pi_1(M) \to \G'$ where
$\G'$ is a subgroup of $\G$, then $\rho(M,\phi') = \rho(M,\phi)$.

\item [(2)] If $\phi$ is trivial (the zero map), then $\rho(M,\phi) = 0$.

\item [(3)] If $M=M_K$ is zero surgery on a knot $K$ and $\phi:\pi_1(M)\to \mathbb{Z}$ is the abelianization, then $\rho(M,\phi)$ is \textbf{denoted by $\boldsymbol{\rho_0(K)}$} and is equal to the integral over the circle of the Levine-Tristram signature function of $K$ ~\cite[Prop. 5.1]{COT2}.

\item [(4)] If $K$ is a slice knot or link and $\phi:M_K\to \G$ ($\G$ PTFA) extends over $\pi_1$ of a slice disk exterior then $\rho(M_K,\phi)=0$ by ~\cite[Theorem 4.2]{COT}.

\end{itemize}

We also make frequent use of the following elementary additivity result for $\rho$-invariants under certain kinds of infections.

Suppose the link $L$ is obtained from the link $R$ by infection, as described above, along a curve $\eta$ using the knot $K$. Let the zero surgeries on $L$, $R$, and $K$ be denoted $M_L$ $M_R$, $M_K$ respectively. Suppose $\phi:\pi_1(M_L)\to \G$ is a map to a metabelian PTFA group $\G$ ($\G^{(2)}=\{e\}$). Since $S^3-K$ is a submanifold of $M_L$, $\phi$ induces  a map on $\pi_1(S^3-K)$. Since the longitude of $K$ lies in $\pi_1(S^3-K)^{(2)}$, it lies in the kernel of $\phi$, so this induced map extends uniquely to a map that we call $\phi_K$  on $\pi_1(M_K)$. For the same reasons, $\phi$ induces a map on $\pi_1(M_R-\eta)$ that extends uniquely to $\phi_R$ on $\pi_1(M_R)$.

\begin{lem}[{\cite[Lemma 2.3]{CHL3}}, see also {\cite[Prop. 3.2]{COT2}}]\label{lem:additivity} In the notation of the previous paragraph,
$$
\rho(M_L,\phi)= \rho(M_R,\phi_R)+\rho(M_K,\phi_K).
$$
Moreover if $\eta\in \pi_1(M_R)^{(1)}$ then $\phi_K$ factors through $\Z$ so either $\rho(M_K,\phi_K)=\rho_0(K)$, or $\rho(M_K,\phi_K)=0$, according as $\phi_R(\eta)\neq 1$ or $\phi_R(\eta)= 1$.
\end{lem}

\section{Zero-th Order L$^{(2)}$-Signatures of Knots and Links}\label{sec:zeroordersigs}

Zero-th order L$^{(2)}$-signatures of a link $J$ will be those associated to \emph{abelian} representations of $\pi_1(M_J)$.

\begin{defn}\label{defn:zeroordersignatures} A \textbf{zero-th order $\boldsymbol{L^{(2)}}$-signature of a link} $J$ is the von Neumann $\rho$-invariant $\rho(M_J,f)$ where $f:\pi_1(M_J)\to A$ and $A$ is a free abelian quotient of $H_1(M_J)$. It will be denoted $\rho^f_0(J)$. If $J$ is a knot, there is a unique zero-th order signature denoted $\rho_0(J)$ (except in the degenerate case that $f=0$, in which case $\rho(M_J,f)=0$).
\end{defn}

The zero-th order signature of a knot $K$ is known to be equal to the average of the Levine-Tristram signatures of the knot ~\cite[Proposition 5.1]{COT2}. The latter are integer-valued signatures, one for each norm one complex number, that are associated to the infinite cyclic covering space of the knot exterior ~\cite[p.242]{L5}. It is a result of Levine that if $K$ is a slice knot then all but a finite number of the Levine-Tristram signatures vanish. It follows that the average, $\rho_0(K)$, vanishes for any slice knot. Similarly, if $J$ is a boundary link then the zero-th order $L^{(2)}$-signature associated to the abelianization map $\pi_1(M_J)\to \Z^m$ is the integral over the $m$-torus of the Levine-Tristram signatures of $J$ (same proof as ~\cite[Proposition 5.1]{COT2}). If $J$ is not a boundary link, various analogues of the Levine-Tristram signatures have been defined by Cooper, Smolinsky and, most notably, Cimasoni-Florens ~\cite{Sm}~\cite{CiF}. Presumably each zero-th order $L^{(2)}$-signature is such an average, but this has not been specifically addressed in the literature.


\section{First-Order L$^{(2)}$-Signatures for Knots}\label{sec:firstorder sigs}

 First-order $L^{(2)}$-signatures of a knot or link $J$ will be associated to \emph{metabelian} representations of $\pi_1(M_J)$. In this section we focus on knots.

 If $K$ is an algebraically slice knot (implying that the zero-th order obstructions are zero) then Casson-Gordon defined ``signature invariants'' for $K$ that also obstructed its being a slice knot ~\cite{CG1}\cite{CG2}. These invariants take the form of \emph{sets} of integers, only \emph{some} of which need vanish for a slice knot. They are associated to metabelian covering spaces of the knot exterior. Further metabelian signatures were defined by C. Letsche, Cochran-Orr-Teichner, and S. Friedl ~\cite{Let}\cite{COT}\cite{Fr2}. All of these metabelian invariants are defined as sets that are parametrized, loosely speaking, by the number of ways in which the zero-th order obstructions vanish.

We extend these (more precisely the $L^{(2)}$ versions) to a larger class. We shall define the first-order signatures for a link $J$ to be von Neumann $\rho$-invariants associated to certain metabelian representations of $\pi_1(M_J)$. Not every metabelian representation need be considered.

Suppose $K$ is an oriented knot and let $G=\pi_1(M_K)$. Note that since the longitude of $K$ lies in $\pi_1(S^3-K)^{(2)}$,
$$
\mathcal{A}_0(K)\equiv G^{(1)}/G^{(2)}\otimes_{\mathbb{Z}[t,t^{-1}]}\mathbb{Q}[t,t^{-1}]
$$
Each submodule $P\subset \mathcal{A}_0(K)$ corresponds to a unique metabelian quotient of $G$,
$$
\phi_P:G\to G/\tilde{P},
$$
by setting
$$
\tilde{P}\equiv \text{kernel}(G^{(1)}\to G^{(1)}/G^{(2)}\to \mathcal{A}_0 \to \mathcal{A}_0/P).
$$
\noindent (Note that $G^{(2)}\subset \tilde{P}$ so $G/\tilde{P}$ is metabelian.) Therefore to any such submodule $P$ there corresponds a real number, the Cheeger-Gromov von Neumann $\rho$-invariant, $\rho(M_K, \phi_P:G\to G/\tilde{P})$.

\begin{defn}\label{defn:firstordersignatures} The \textbf{first-order $\boldsymbol{L^{(2)}}$-signatures of a knot} $K$ are the real numbers
$\rho(M_K, \phi_P)$ where $P$ is an isotropic submodule of $\mathcal{A}_0(K)$ with respect to $\mathcal{B}\ell_0^K$.
\end{defn}

If $\Delta_K(t)=1$ then $\mathcal{A}_0(K)=0$ and $G^{(1)}=G^{(2)}$. It follows that the only first-order signature of $K$ is actually $\rho_0(K)$ which is zero since $K$ has zero classical signatures almost everywhere. The set of first-order signatures of a knot is an isotopy invariant of the knot. None of the individual first-order signatures is a concordance invariant.

Suppose $K$ is a slice knot, $\Delta$ is a slice disk for $K$ and $V=B^4-\Delta$. Then $\partial V=M_K$. Set
\begin{equation}\label{eq:defP}
P_\Delta=\text{ker}\left(H_1(M_K;\Q[t,t^{-1}])\to H_1(V;\Q[t,t^{-1}])\right).
\end{equation}
Then it is well-known that $P_\Delta$ is a Lagrangian for $\mathcal{B}\ell_0^K$. In this case we say that the Lagrangian $P_\Delta$ \textbf{corresponds to the slice disk}, $\Delta$.

\
\begin{thm}[{~\cite[Thms 4.2, 4.4]{COT}}]\label{thm:firstordersigs=0} If $K$ is a slice knot then, for any Lagrangian $P_\Delta$ that corresponds to a slice disk $\Delta$, the corresponding first-order $L^{(2)}$-signature of $K$ vanishes. Thus if $K$ is a slice knot then the set of all first-order signatures corresponding to Lagrangians contains $0$.
\end{thm}

The signatures of this theorem are denoted \textbf{metabelian $\boldsymbol{L^{(2)}}$-signatures}, and first appeared in ~\cite{COT}. Their relationship with Casson-Gordon invariants and other ``metabelian signatures'' was beautifully explained by S. Friedl ~\cite{Fr2}. The analogue of this theorem is also well-known for the other metabelian signatures. Indeed, there are even finer versions of first-order $L^{(2)}$-signatures obtained by projecting $\mathcal{A}_0(K)/P$ onto any proper quotient (closed under the involution) (see for example ~\cite{Fr2}). These finer versions do not yet play a role in second-order signatures, so they are omitted here.

It follows from Theorem \ref{thm:firstordersigs=0} that the set of first-order signatures of a knot obstructs its being a slice knot. For this purpose alone it is not necessary to consider first-order signatures that correspond to isotropic submodules that are not Lagrangians. However, for second-order signatures, we seem to need this general notion. Note that $P=0$ is always isotropic and never a Lagrangian (unless $\Delta_K(t)=1$), so we give a special name to the signature corresponding to this case.

\begin{defn}[{\cite{CHL3}\cite[Section 4]{CHL4}}]\label{defn:rho1} $\boldsymbol{\rho^1(K)}$ is the first-order $L^{(2)}$-signature given by the Cheeger-Gromov invariant $\rho(M_K, \phi:G\to G/G^{(2)})$.
\end{defn}

\textbf{Open Problem}: Find methods to calculate the metabelian signature $\rho^1(K)$.

\begin{ex}\label{ex:knotfirst-ordersigs} Consider the knot $K$ in Figure~\ref{fig:exampleknothighersigs}, which is a genus one algebraically slice knot (whose Alexander polynomial is not $1$). For such a knot, any isotropic submodule $P$ must have $\Q$-rank $0$ or $1$. In the former case, $P=0$, and in the latter case $P$ is a Lagrangian. It is easy to see that such a knot has precisely two Langrangians.  Thus $K$ has precisely $3$ first-order signatures, two corresponding to the Lagrangians $P_1$ and $P_2$ and the third corresponding to $P_3=0$.
\begin{figure}[htbp]
\setlength{\unitlength}{1pt}
\begin{picture}(200,160)
\put(0,0){\includegraphics[height=150pt]{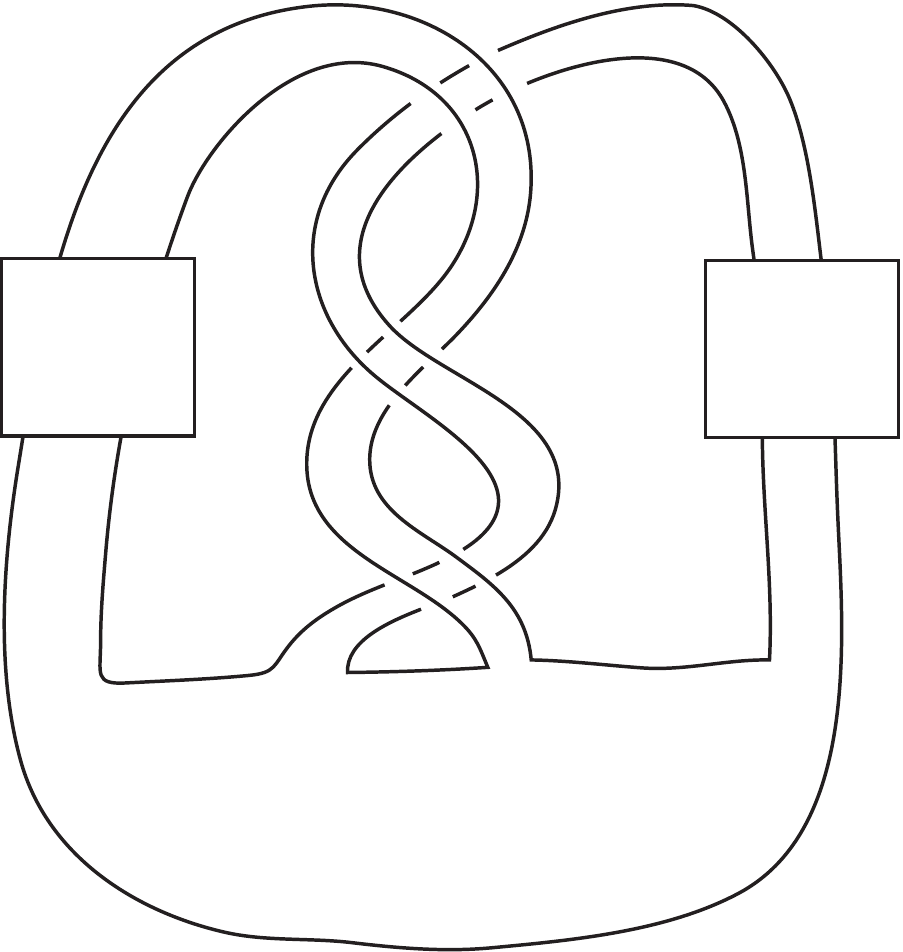}}
\put(-50,70){$K=$}
\put(7,92){$J_1$}
\put(122,92){$J_2$}
\end{picture}
\caption{A genus one algebraically slice knot $K$}\label{fig:exampleknothighersigs}
\end{figure}
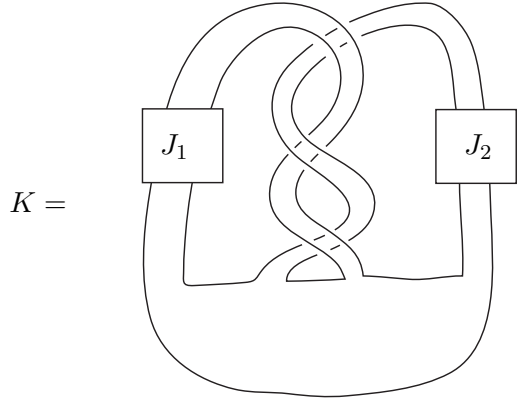
Using Lemma~\ref{lem:additivity}, it was shown in ~\cite[Example 3.3]{CHL3} that
 that the set of first-order signatures of $K$ is
 $$
\{\rho^1(9_{46})+\rho_0(J_1)+\rho_0(J_2),\rho_0(J_1),\rho_0(J_2)\}. $$
Here $9_{46}$ is the ribbon knot obtained by setting $J_1=J_2=$unknot. So far we have been unable to calculate $\rho^1(9_{46})$ but certainly it is true that for any $J_i$ such that $\rho_0(J_i)\neq 0$ and $\rho_0(J_1)+\rho_0(J_2)\neq -\rho^1(9_{46})$ (which may be guaranteed for example by making $\rho_0(J_i)$ sufficiently large) then \textbf{none} of the first-order signatures of $K$ is zero.
\end{ex}

\begin{ex}\label{ex:figeight} Consider the knot $K$ in Figure~\ref{fig:examplehighersigsfigeight} which is of order two in the algebraic concordance group. A genus one knot that is \emph{not} zero in the rational algebraic concordance group (that is, there is no Lagrangian) has precisely one first-order signature, namely $\rho^1(K)$, since any \emph{proper} submodule $P$ of the rational Alexander module satisfying $P\subset P^\perp$ would have to be a Lagrangian. Using Lemma~\ref{lem:additivity} and the amphichirality of the figure-eight knot, it was shown in ~\cite[Example 3.5]{CHL3} that
$$
\rho^1(K)=2\rho_0(J).
$$
\begin{figure}[htbp]
\begin{picture}(185,118)
\put(10,0){\includegraphics{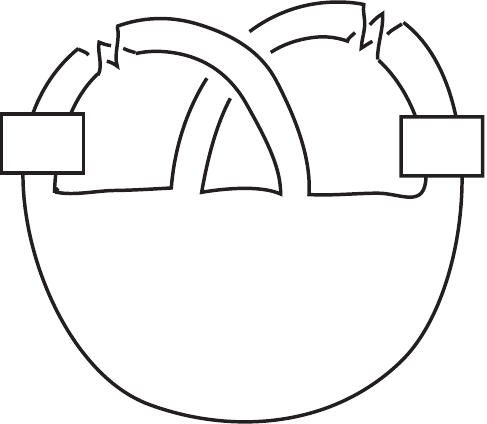}}
\put(17,77){$J$}
\put(135,77){$J$}
\end{picture}
\caption{} \label{fig:examplehighersigsfigeight}
\end{figure}
\end{ex}

\begin{ex}\label{ex:nonzerofirst-ordersigs} Consider the fully amphichiral ribbon knot $8_9$, pictured on the left-hand side of Figure~\ref{fig:8_9} (a ribbon move is shown by the dotted arc)~\cite{Lam}. Since the Alexander polynomial of $8_9$ is the product if 2 distinct primes, the Alexander module is cyclic. Thus $\mathcal{A}_0(8_9)$ has precisely $3$ proper submodules ~\cite[p.279]{Ka3}. Two are Lagrangians that correspond to slice disks and the third is $P=0$. It was shown in ~\cite[Example 4.4]{CHL4} that each of the corresponding first-order signatures is zero.
\begin{figure}[htbp]
\setlength{\unitlength}{1pt}
\begin{picture}(200,160)
\put(-45,0){\includegraphics{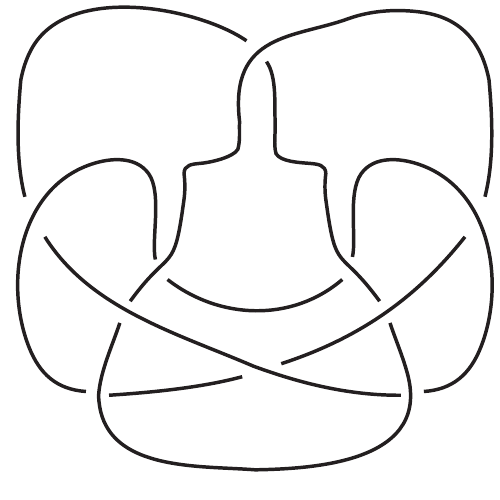}}
\put(145,0){\includegraphics{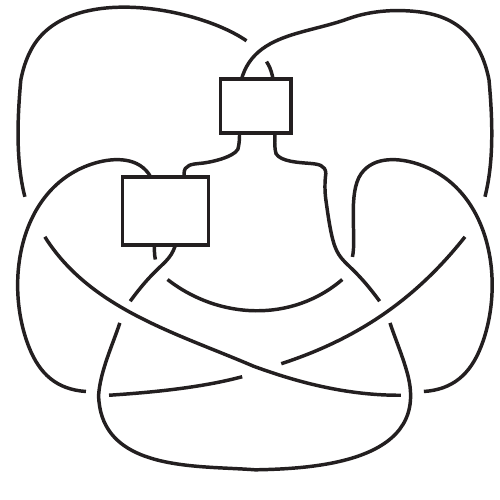}}
\put(-70,80){$8_9~=$}
\put(22,136){....}
\put(212,111){$K_1$}
\put(186,80){$K_1$}
\put(125,80){$K=$}
\end{picture}
\caption{}\label{fig:8_9}
\end{figure}

By contrast, consider the family of algebraically slice knots, $K$, shown on the right-hand side of Figure~\ref{fig:8_9}. Since $K$ has the same Alexander module and classical Blanchfield linking form as $8_9$, it also has $3$ first-order signatures. But the values of its first-order signatures are altered by the zero-th order signature of the knot $K_1$. Using Lemma~\ref{lem:additivity}, it was shown in ~\cite[Example 4.6]{CHL4} that the first-order signatures of $K$ are $\{2\rho_0(K_1),\rho_0(K_1),2\rho_0(K_1)\}$. Thus if $\rho_0(K_1)\neq 0$ then all of the first-order signatures of $K$ are \emph{non-zero}.
\end{ex}

\section{Derivatives of Knots}\label{sec:derivatives}

In this section we define the (partial) derivative of a knot with respect to a metabolizer of its Seifert form. This formalizes notions that have been implicit in the subject of knot concordance since the early work of Levine. The dual notion of an antiderivative plays a less central role and is discussed in Section~\ref{sec:antiderivatives}.

Suppose $K$ is a knot and $\Sigma$ is a genus $g$ Seifert surface for $K$. A \textbf{metabolizer}, $\mathfrak{m}$, for $K$ is a rank $g$ \emph{summand} of $H_1(\Sigma;\mathbb{Z})\cong \mathbb{Z}^{2g}$ on which the Seifert form vanishes. If $K$ is an algebraically slice knot then any Seifert surface admits a metabolizer, but there may be many (even an infinite number of) distinct metabolizers. Since the Seifert form vanishes on $\mathfrak{m}$, the intersection form vanishes on $\mathfrak{m}$. Therefore we can realize $\mathfrak{m}$ by a set of $g$ disjoint oriented simple closed curves $\{J_1,...,J_g\}$ that is a basis for $\mathfrak{m}\subset H_1(\Sigma)$. In this way (as is well known) any metabolizer can be realized (though not in a unique way).

\begin{defn}\label{defn:geomderivative} If $K$ is an algebraically slice knot, $\Sigma$ is a genus $g$ Seifert surface for $K$, and $\mathfrak{m}$ is a metabolizer for the Seifert form on $H_1(\Sigma)$, then a \textbf{derivative of $\boldsymbol{K}$ with respect to $\boldsymbol{\mathfrak{m}}$} is a $g$ component oriented link $J$ embedded in $\Sigma$ where $\{[J_1],...,[J_g]\}$ is a basis of $\mathfrak{m}$. It is denoted by $\frac{\partial K}{\partial \mathfrak{m} }$.
\end{defn}

\begin{ex}\label{ex:arbgenusonederiv} Any genus one algebraically slice knot $K$ admits a Seifert matrix of the form
$$
\left(\begin{matrix}0 & \ell \cr \ell+1 & t
\end{matrix}\right)
$$
and moreover is isotopic to the form shown on the left-hand side of Figure~\ref{fig:arbgenus1AS} where $\ell$ is the number of full twists between the bands, $t$ is the number of full twists of the right-hand band, and $L$ is a $2$-component string link with linking number zero into which the two bands are tied.

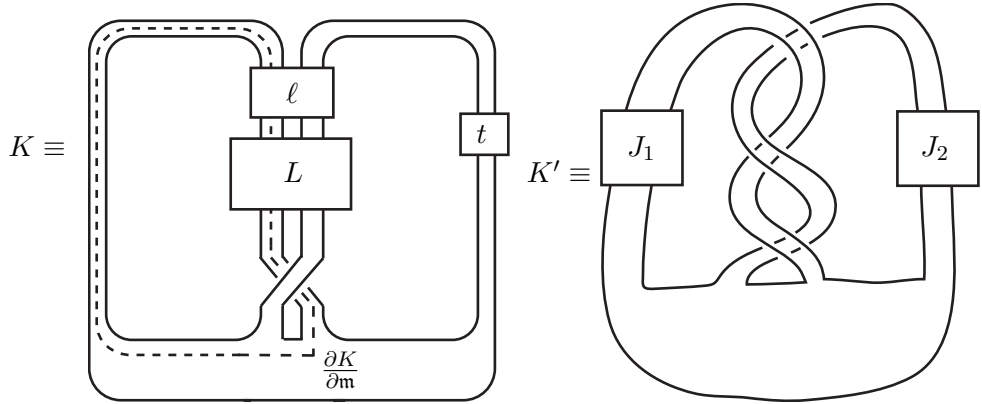
\begin{figure}[htbp]
\setlength{\unitlength}{1pt}
\begin{picture}(327,151)
\put(0,0){\includegraphics{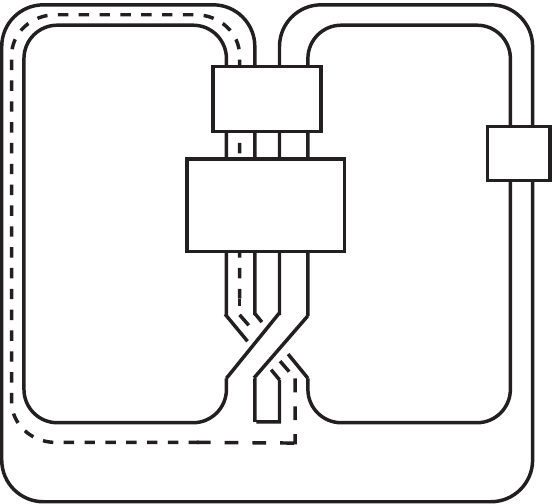}}
\put(194,0){\includegraphics{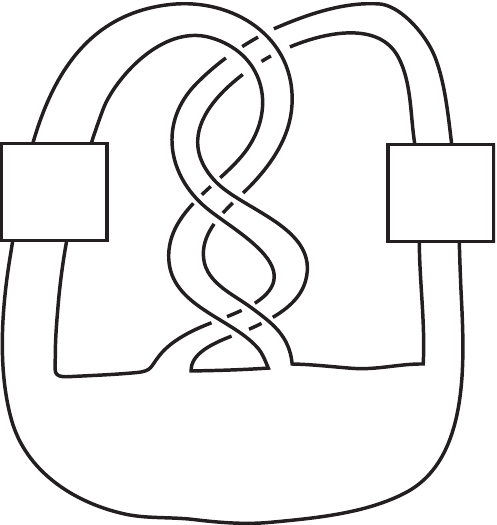}}
\put(147,98){$t$}
\put(75,113){$\ell$}
\put(74,83){$L$}
\put(88,9){$\frac{\partial K}{\partial \mathfrak{m} }$}
\put(204,93){$J_1$}
\put(315,93){$J_2$}
\put(166,83){$K'\equiv$}
\put(-30,93){$K\equiv$}
\end{picture}
\caption{An arbitrary genus one algebraically slice knot $K$; and a very special example $K'$}\label{fig:arbgenus1AS}
\end{figure}

If $\mathfrak{m}$ is the metabolizer generated by the zero-twisted band then $\frac{\partial K}{\partial \mathfrak{m} }$ is merely $L_1$, the first component knot of the link that is the closure of $L$, as shown by the dashed curve. If $t=0$ then $\frac{\partial K}{\partial \mathfrak{m}' }$ for the other metabolizer $\mathfrak{m}'$ is a curve of the knot type of $L_2$ (going over the other band). In general the $\frac{\partial K}{\partial \mathfrak{m}' }$ will be a knot that goes over both bands a number of times that depends on $t$ and $\ell$. A very special case of this situation is shown by the knot $K'$ on the right-side of Figure~\ref{fig:arbgenus1AS}. In this case, $L=\{J_1,J_2\}$ is a \textbf{split} link and the resulting knot $K'$ is a satellite knot. This is one of the ways in which the present paper is an improvement over the authors previous work ~\cite{CH3} in which only knots similar to $K'$ are handled.
\end{ex}

\begin{ex}\label{ex:ribbon} Suppose $K$ is a ribbon knot that bounds a ribbon disk $D$ in $S^3$ with $g$ ribbon singularities. To $D$ we may associate a Seifert surface $\Sigma$ by locally de-singularizing each ribbon singularity. For each ribbon singularity, choose a small sub-disk of $D$ containing the corresponding slit. The $g$-component trivial link formed by the boundaries of these disks is a derivative of $K$. Hence any ribbon knot admits a derivative that is a trivial link. In fact it can be shown that any Seifert surface for a ribbon knot, after hollow handle enlargements, admits a metabolizer represented by a trivial link.
\end{ex}

More examples wherein the derivatives are links are given below.

It is a serious problem that for a higher-genus algebraically slice knot there are usually an infinite number of metabolizers. It is better to consider Lagrangians, which are quite often finite in number.

\begin{defn}\label{def:represents} Suppose $P\subset \mathcal{A}_0(K)$ is a Lagrangian. The metabolizer $\mathfrak{m}$ \textbf{represents P} if the image of $\mathfrak{m}$ under the map
$$
H_1(\Sigma;\mathbb{Z})
\overset{id\otimes 1}{\hookrightarrow}H_1(\Sigma;\mathbb{Z})\otimes \mathbb{Q}\overset{i_*}\twoheadrightarrow \mathcal{A}_0(K)
$$
spans $P$ as a $\Q$-vector space. This is sometimes denoted $\mathfrak{m}_P$. (Here, to define $i_*$, we have in mind fixing a lift of $\Sigma$ to the infinite cyclic cover. It is easy to see that the definition is independent of this choice.)
\end{defn}

\begin{lem}\label{lem:lagrangianisrep} Every Lagrangian is represented by some metabolizer.
\end{lem}

This result is surprisingly difficult (for us) to prove and the casual reader might want to skip the proof which we postpone until the end of this section. In the case of a genus one knot the proof is much easier. If $K$ is a genus one slice knot and $P$ corresponds to a slice disk, then the proof is quite short. For then, by the standard argument, there exists a simple closed curve, $J$, on any genus one Seifert surface that dies in the rational Alexander module of the exterior of the slice disk, and hence lies in $P$. We can assume that $\Delta_K(t)\neq 1$. Since rank$_\Q P=1$, we are done unless $J$ were zero in $\mathcal{A}_0(K)$, which contradicts the well-known fact $H_1(\Sigma;\Q)$ spans the $2$-dimensional vector space $\mathcal{A}_0(K)$.

Note that the definition of a derivative of a knot does not require one to choose a symplectic basis for $H_1(\Sigma;\mathbb{Z})$. However, since $\mathfrak{m}$ is a Lagrangian subspace of $H_1(\Sigma;\mathbb{Z})$ with respect to the intersection form we can extend any ordered basis $\{a_1,...,a_g\}$ of $\mathfrak{m}$  to a symplectic basis $\{a_1,...,a_g,b_{1},...b_{g}\}$ for $H_1(\Sigma;\mathbb{Z})$ (in our notation $a_i$ is intersection dual to $b_{i}$). Moreover we can realize  these by oriented simple closed curves $\{J_1,...,J_g,J_{g+1},...,J_{g+g}\}$ on $\Sigma$.  This induces a ``disk-band'' form on $\Sigma$ as shown in Figure \ref{fig:funkytwisty} below.

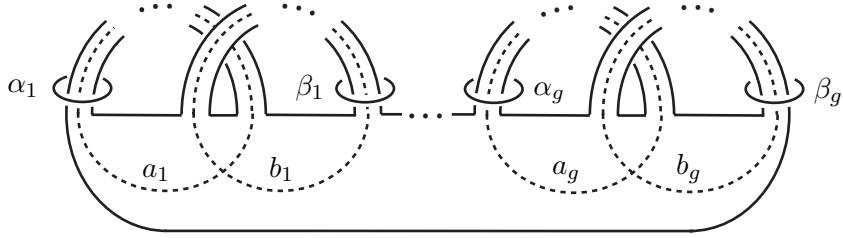
\begin{figure}[htbp]
\setlength{\unitlength}{1pt}
\begin{picture}(165,100)
\put(-60,0){\includegraphics{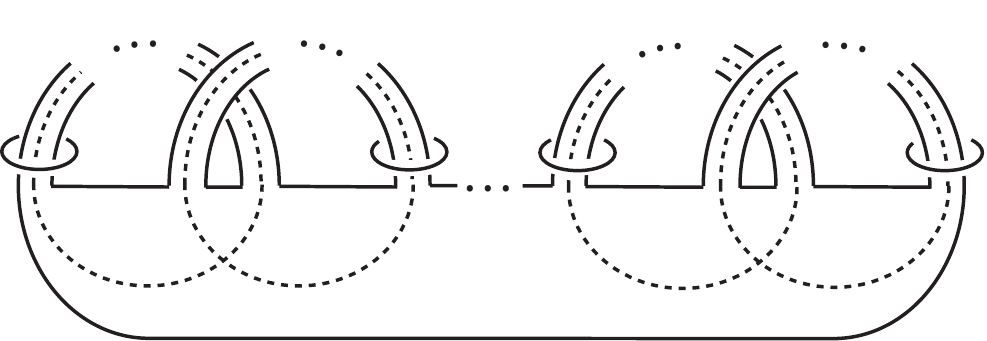}}
\put(-77,53){$\alpha_1$}
\put(-26,21){$a_1$}
\put(22,21){$b_1$}
\put(32,52){$\beta_1$}
\put(122,52){$\alpha_g$}
\put(129,22){$a_g$}
\put(176,22){$b_g$}
\put(228,51){$\beta_g$}
\end{picture}
\caption{Disk-band form for $\Sigma$}\label{fig:funkytwisty}
\end{figure}

\begin{prop}\label{prop:lagrangprops} Suppose $P\subset \mathcal{A}_0(K)$ is a Lagrangian. Then for any Seifert surface $\Sigma$, any metabolizer $\mathfrak{m}$ representing $P$ and any symplectic basis $\{a_1,..,a_{g},b_1,..b_g\}$ of $H_1(\Sigma)$ with $\{a_1,...,a_g\}$ a basis for $\mathfrak{m}$ we have
\begin{itemize}
\item [1.] $\{a_1,...,a_g\}$  spans $P$ in the rational vector space $\mathcal{A}_0(K)$.
\item  [2.] $\{\phi(\alpha_{1}),...,\phi(\alpha_{g})\}$ spans $\mathcal{A}_0(K)/P$, where $\{\alpha_{1},..,\alpha_{g},\beta_1,..,\beta_g\}$ is the basis of $H_1(S^3-\Sigma)$ that is dual to $\{a_1,..,a_{g},b_1,..,b_g\}$ under linking number in $S^3$ and

$$
\phi: H_1(S^3-\Sigma;\Z)\hookrightarrow H_1(S^3-\Sigma;\Z)\otimes \Q\overset{i_*}\twoheadrightarrow \mathcal{A}_0(K).
$$
\end{itemize}
\end{prop}

\begin{proof} Property $1$ is immediate from Definition~\ref{def:represents}. It is well known that
$$
H_1(S^3-\Sigma;\Z)\otimes \Q\overset{i_*}\twoheadrightarrow \mathcal{A}_0(K),
$$
is surjective. Hence $\{\alpha_{1},..,\alpha_{g},\beta_1,..,\beta_g\}$  spans $\mathcal{A}_0(K)$ under $\phi$. Furthermore, one can easily see from the example shown in Figure \ref{fig:tubes} that each $\beta_{i}$ bounds a disk that hits $K$ twice. Then, by tubing, one sees that $\beta_{i}$ bounds a punctured torus.

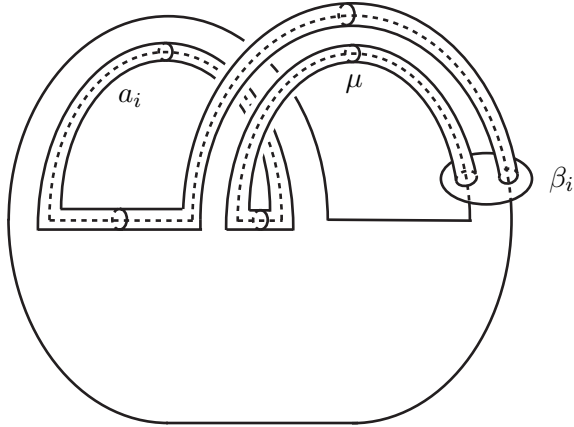
\begin{figure}[htbp]
\setlength{\unitlength}{1pt}
\begin{picture}(165,170)
\put(-20,0){\includegraphics{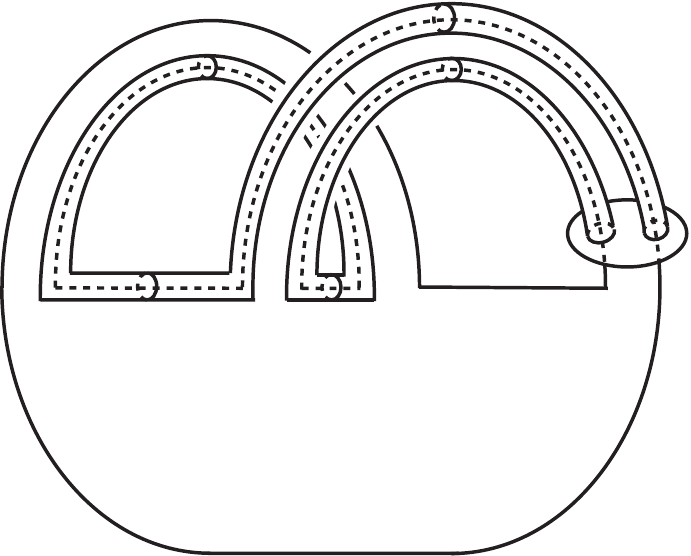}}
\put(22,122){$a_i$}
\put(108,128){$\mu$}
\put(185,90){$\beta_i$}
\end{picture}
\caption{$\beta_{i}$ bounds a punctured torus}\label{fig:tubes}
\end{figure}

This torus admits a symplectic basis wherein one curve is a meridian, $\mu$, of $K$ and the other is $a_{i}$. This illustrates the fact that $[\beta_{i}]=(t-1)[a_{i}]$ in $\mathcal{A}_0(K)$. Thus the sets $\{a_{1},...,a_{g}\}$ and $\{\phi(\beta_{1}),...,\phi(\beta_{g})\}$ generate the same subspace of $\mathcal{A}_0(K)$, namely $P$. It follows that $\{\phi(\alpha_{1}),...,\phi(\alpha_{g})\}$ spans $\mathcal{A}_0(K)/P$.
\end{proof}

If $P\subset \mathcal{A}_{0}(K)$ is a Lagrangian and $\mathfrak{m}$ is a metabolizer representing $P$, then we can use Proposition \ref{prop:lagrangprops} to take the viewpoint that the derivative $\frac{\partial K}{\partial\mathfrak{m}}$ comes equipped with a \textbf{canonical epimorphism} $f:\pi_1(M_J)\to A$, where $A$ is a free abelian quotient of $H_1(M_J)$, defined as follows. To a meridian $\mu_i$ of $J_i$ we associate the meridian, $\alpha_i$, of the band on which $J_i$ lies (see Figure \ref{fig:funkytwisty}) and set $f(\mu_i)=\phi(\alpha_i)$.

Note that $A$, the image of $\phi$, is a free abelian group. We often think of $f$ as a map to $\Z^d$ ($d=\frac{1}{2}\text{deg}\Delta_K(t)$) but distinguish $f$ only up to post-composition with an isomorphism. We say that $\frac{\partial K}{\partial \mathfrak{m} }=(J,f)$. Note that $f$ is trivial if and only if $\Delta_K(t)=1$. If the genus of $\Sigma$ is one then $J$ is a knot and in this case the map $f$ will merely be the abelianization $\pi_1(M_J)\to \Z$ unless $\Delta_K(t)=1$ in which case it will be the zero homomorphism. Note that \emph{any} knot or link $J$ can arise as a derivative of a slice Alexander polynomial $1$ knot (see Section~\ref{sec:antiderivatives}). In these cases, however, the map $f$ is the zero map. Thus it is necessary to include the map $f$ in the data, especially for higher genus examples.


Recall Gilmer's philosophy that the Casson-Gordon invariants of a knot $K$ are related to the classical signatures and nullities of knots on the Seifert surface for $K$ (~\cite[Corollaries 0.2, 0.3]{Gi3} although Gilmer informs us that those Corollaries are invalid due to a gap in the paper). The following results, proved in Section~\ref{sec:nullbordisms}, are the analogous results for first-order $L^{(2)}$-signatures.

\begin{prop}\label{prop:firstrho=rhozeroofder} Suppose that $P$ is a Lagrangian for $K$ and $(J,f)=\frac{\partial K}{\partial \mathfrak{m} }$ is a $c$-component link where $\mathfrak{m}$ represents $P$. If the first-order signature of $K$ corresponding to $P$ is denoted $\rho(M_K,\phi_P)$ then
$$
|\rho(M_K,\phi_P)-\rho^f_0(J)|\leq c-1-\eta(J,f).
$$
\end{prop}

\begin{cor}\label{cor:firstrho=rhozeroofder} Suppose that $P$ is a Lagrangian for $K$ and $(J,f)=\frac{\partial K}{\partial \mathfrak{m} }$ is a link of maximal Alexander nullity (for example a knot). Then the first-order signature of $K$ corresponding to $P$ is equal to $\rho^f_0(\frac{\partial K}{\partial \mathfrak{m} })$, that is, the zero-th order signature of the derivative with respect to $\mathfrak{m}$.
\end{cor}
Combining Proposition~\ref{prop:firstrho=rhozeroofder} with Theorem~\ref{thm:firstordersigs=0} we get the following generalization of Cooper's Theorem.
\begin{cor}\label{cor:sliceimpliesrhobound} If $K$ is a slice knot, $P$ is a Lagrangian corresponding to a slice disk and $(J,f)=\frac{\partial K}{\partial \mathfrak{m} }$ is a $c$-component link where $\mathfrak{m}$ represents $P$, then
$$
|\rho^f_0(J)|\leq c-1-\eta(J,f).
$$
\end{cor}

\begin{ex}\label{ex:twistknots} Let $K$ be a twist knot, the boundary of the genus one surface shown in Figure~\ref{fig:twistknot}. It is easy to see that $K$ is algebraically slice precisely when $4t+1=m^2$ for some positive integer $m$. Casson-Gordon showed that such knots are in fact not slice unless $t=0$ or $2$ ~\cite{CG1}. We can reprove this result from the present point of view. This point of view was well-known to experts. It is an excellent illustration of what we are trying to generalize. Assuming $4t+1=m^2$, the genus one surface shown admits two metabolizers generated by $\{(1,\frac{1\pm m}{2})\}$ respectively. Consider the case of $(1,\frac{1+m}{2})$. Then $J=\frac{\partial K}{\partial \mathfrak{m} }$ has the form of the knot shown in Figure~\ref{fig:twistknot} and $J$ has the knot type of the $(n,1-n)$-torus knot  where $n=\frac{1+m}{2}$. Note that unless $n=1$ or $2$ (i.e. $t=0$ or $2$), $\frac{\partial K}{\partial \mathfrak{m} }$ is a non-trivial torus knot and hence has non-trivial zero-th order signature. The case of $(1,\frac{1-m}{2})$ is similar. It follows from Theorem~\ref{thm:CooperThm} that $K$ is not a slice knot except in these two cases. Alternatively, it follows from Corollary~\ref{cor:firstrho=rhozeroofder} that the first-order signatures of $K$ corresponding to Lagrangians are non-zero, and then follows from Theorem~\ref{thm:firstordersigs=0} that $K$ is not slice.
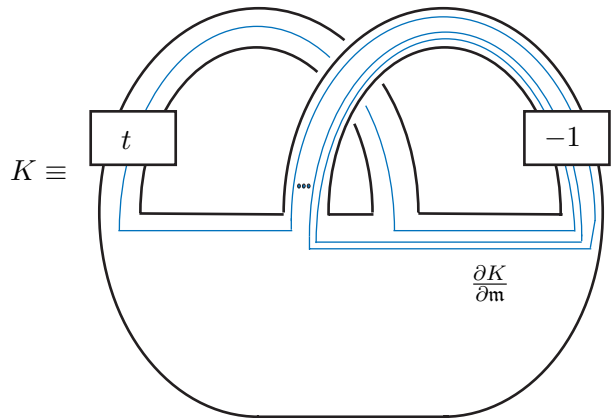
\begin{figure}[htbp]
\setlength{\unitlength}{1pt}
\begin{picture}(327,151)
\put(80,0){\includegraphics{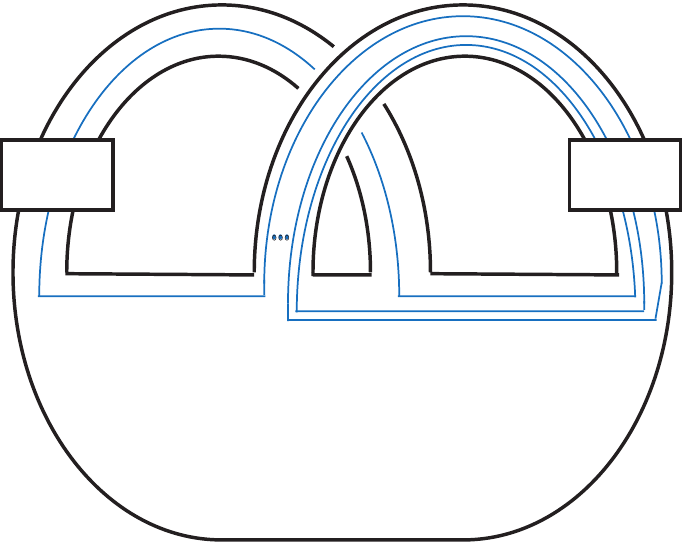}}
\put(252,103){$-1$}
\put(92,102){$t$}
\put(224,48){$\frac{\partial K}{\partial \mathfrak{m} }$}
\put(50,90){$K\equiv$}
\end{picture}
\caption{A derivative of a twist knot}\label{fig:twistknot}
\end{figure}

\end{ex}

\begin{ex}\label{ex:changerho}This example shows that Proposition~\ref{prop:firstrho=rhozeroofder} is close to the best possible result. It also shows that, even for a fixed Lagrangian $P$, derivatives with respect to metabolizers representing $P$ can vary greatly. Suppose $W$ is a $2$-component string link whose second component is unknotted and which has linking number zero. Then the knot $K$ shown in Figure~\ref{fig:stupidhandles} (for fixed $L$, $\ell$ and $t$) is isotopic, independent of  $W$, to the genus one knot shown on the left-hand side of Figure~\ref{fig:arbgenus1AS}.
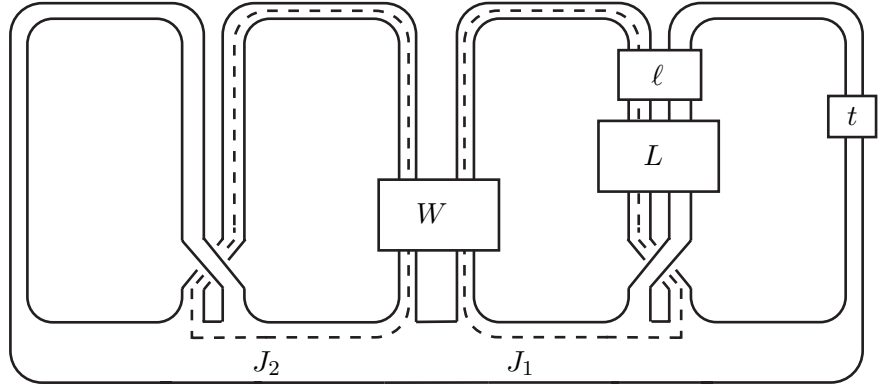
\begin{figure}[htbp]
\setlength{\unitlength}{1pt}
\begin{picture}(205,170)
\put(-70,0){\includegraphics{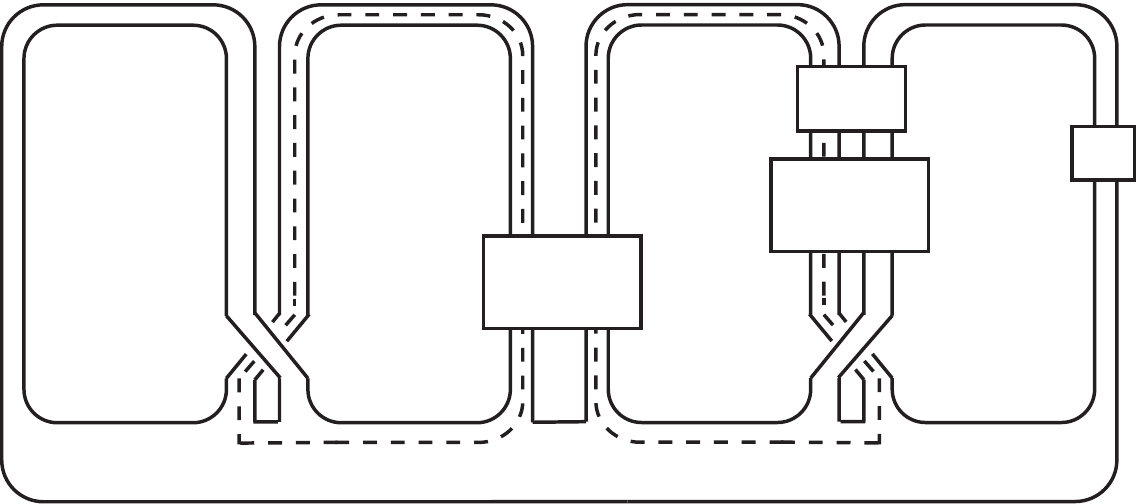}}
\put(22,5){$J_2$}
\put(118,5){$J_1$}
\put(84,61){$W$}
\put(170,83){$L$}
\put(173,113){$\ell$}
\put(247,98){$t$}
\end{picture}
\caption{Trade-off between nullity and signature}\label{fig:stupidhandles}
\end{figure}
The link $J_W=\{J_1,J_2\}$ is the basis of a metabolizer $\mathfrak{m}_w$ on the obvious genus two surface $\Sigma$ shown in Figure~\ref{fig:stupidhandles}. Thus $J_W=\frac{\partial K}{\partial \mathfrak{m}_w }$. Each $\mathfrak{m}_w$ represents the same Lagrangian $P$ (independent of the choice of $W$) since $J_2$ is trivial in the Alexander module (it bounds a surface in the exterior of $\Sigma$). In each case the associated epimorphism $f:\pi_1(J_W)\to\Z$ sends the meridian of $J_1$ to $1$ and the meridian of $J_2$ to zero. Note that the first-order signature of $K$ corresponding to $P$ is a fixed well-defined real number. Moreover, by using the genus one Seifert surface, and applying Corollary~\ref{cor:firstrho=rhozeroofder} we see that this number is equal to $\rho_0(L_1)$. However, one can check that different choices of $W$ lead to different values of the Alexander nullity of the derivative and to different values of the zero-th order signature of the derivative. Specifically if $W$ is the trivial link then $\eta(J_W,f)=1$ while $\rho^f_0(J_W)=\rho_0(L_1)$, whereas if $W$ is a Whitehead link then $\eta(J_W,f)=0$ and $\rho^f_0(J_W)=\rho_0(L_1)+1$. Therefore one sees that there is a trade-off between the nullity and the signature and that both must be included to properly estimate the first-order signature of $K$ corresponding to $P$.
\end{ex}

\textbf{Open Problem}: If $\Sigma$ is a Seifert surface for a (smoothly) slice knot does there exist a choice $\frac{\partial K}{\partial \mathfrak{m} }$ that is a link of maximal Alexander nullity? a trivial link?

\begin{ex}\label{ex:derivsofhighergenus} Consider the family of knots, K, shown in Figure~\ref{fig:derivsgenustwo} with the obvious genus two Seifert surface. The integers $\ell_i$ indicate full-twists between the bands. $\mathbb{B}$ symbolizes a $2$-component string link whose components are parallel copies of the knotted arc $B$ so that the bands are tied into parallel copies of the knot $B$. Here $L=\{L_1,L_2\}$ and $\mathbb{L}=\{\mathbb{L}_1,\mathbb{L}_2\}$ are (string) links of linking number zero. We make the following restrictions. Assume that $\ell_1>\ell_2\geq 1$. Assume that $L_1$  and $\mathbb{L}_2$ are algebraically slice knots and assume that $\rho_0(L_2)>0$, $\rho_0(\mathbb{L}_1)>0$ and $\rho_0(B)\geq 0$. We shall compute representatives of all of the derivatives of $K$. Each will be a $2$-component link with maximal Alexander nullity (indeed a boundary link). Only one Lagrangian will have a non-zero associated first-order signature.
\begin{figure}[htbp]
\setlength{\unitlength}{1pt}
\begin{picture}(327,151)
\put(0,0){\includegraphics{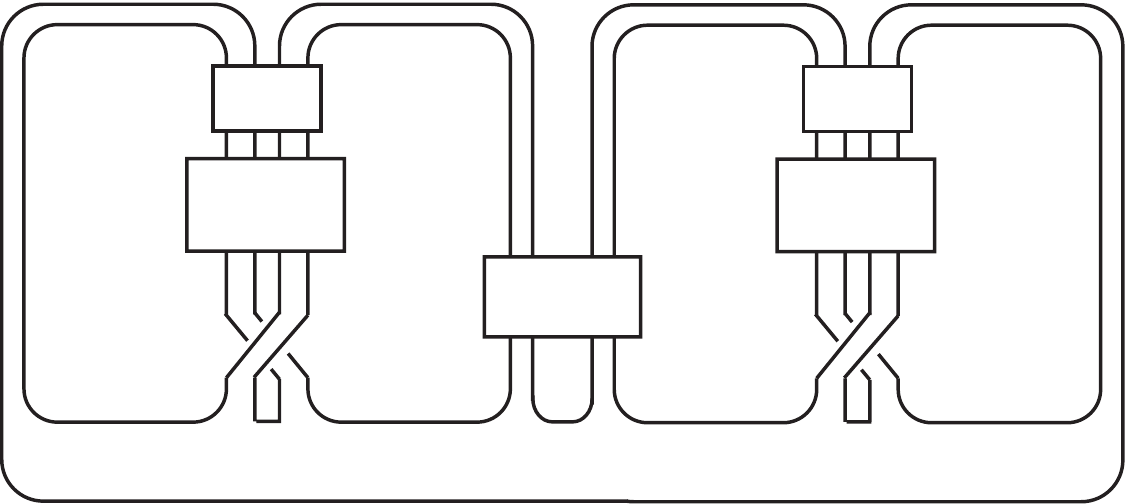}}
\put(243,83){$\mathbb{L}$}
\put(243,113){$\ell_2$}
\put(74,112){$\ell_1$}
\put(158,56){$\mathbb{B}$}
\put(71,83){$L$}
\end{picture}
\caption{A family of genus two knots, $K$}\label{fig:derivsgenustwo}
\end{figure}
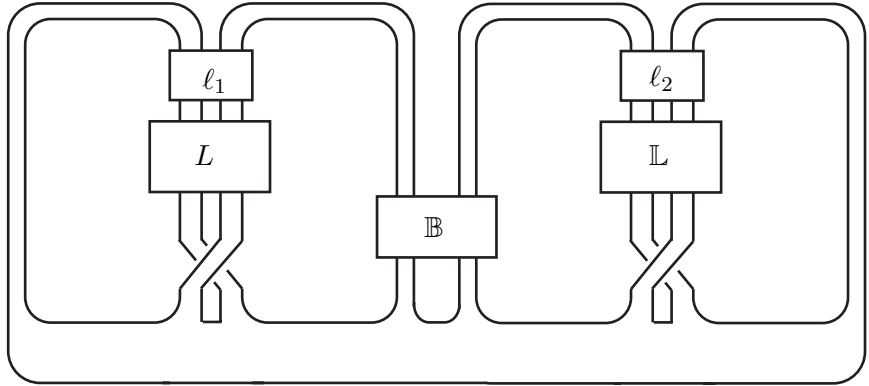

None of $\mathbb{B}$, $L$ or $\mathbb{L}$ affects the Seifert matrix for $K$ so $K$ has the Alexander module and Blanchfield form of a connected sum of genus one knots ($K^{\ell_1}, K^{\ell_2}$) of the type shown on the left-hand side of Figure~\ref{fig:arbgenus1AS} (with $t=0$). Moreover since $\ell_1\neq \ell_2$ the Alexander polynomials of these genus one knots are co-prime. Thus
$$
\mathcal{A}_0(K)\cong \frac{\Q[t,t^{-1}]}{\langle p(t)p(t^{-1})\rangle}\oplus \frac{\Q[t,t^{-1}]}{\langle q(t)q(t^{-1})\rangle}
$$
where $p(t)$ and $q(t)$ are distinct primes. It follows that any Lagrangian of $\mathcal{A}_0(K)$ is a direct sum $P_i\oplus Q_j$ of Lagrangians for $K^{\ell_1}$ and $K^{\ell_2}$ respectively. Therefore $K$ has precisely four Lagrangians of the form $P_i\oplus Q_j$ obtained from $P_1=\langle (p(t),0)\rangle$, $P_2=\langle (p(t^{-1}),0)\rangle,$ $Q_1=\langle (0,q(t))\rangle$ and $Q_2=\langle (0,q(t^{-1}))\rangle$. The rank one subspaces $P_1,P_2,Q_1,Q_2$ have representative curves on the genus two surface in Figure~\ref{fig:derivsgenustwo} that traverse the first, second, third and fourth bands respectively. Let $\mathfrak{m}_{ij}, 1\leq i\leq 2, 1\leq j\leq 2$ be the rank $2$ metabolizers that represent $P_i\oplus Q_j$ obtained by using the appropriate choices of the band curves. Then the four derivatives $J_{ij}=\frac{\partial K}{\partial \mathfrak{m}_{ij} }$ are as shown in Figure~\ref{fig:4derivs}.
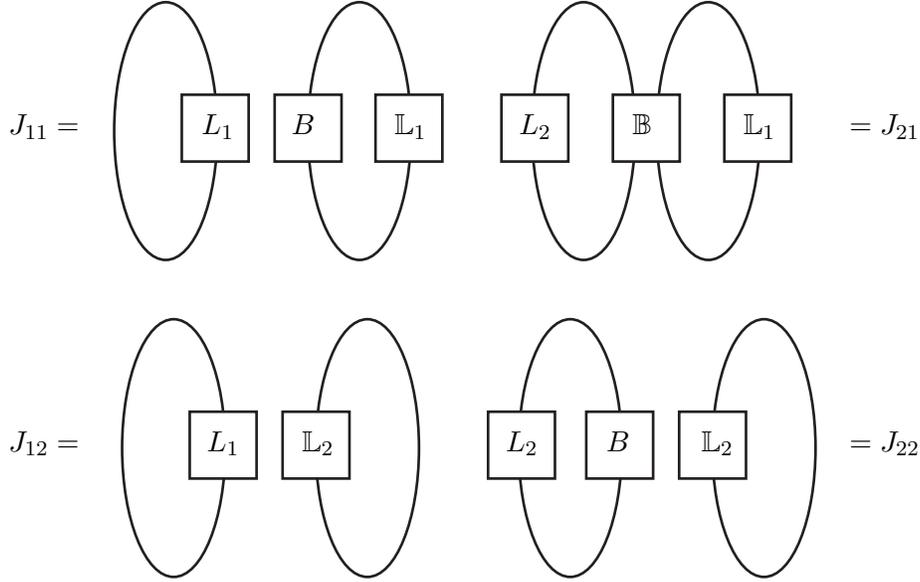
\begin{figure}[htbp]
\setlength{\unitlength}{1pt}
\begin{picture}(200,230)
\put(-53,120){\includegraphics{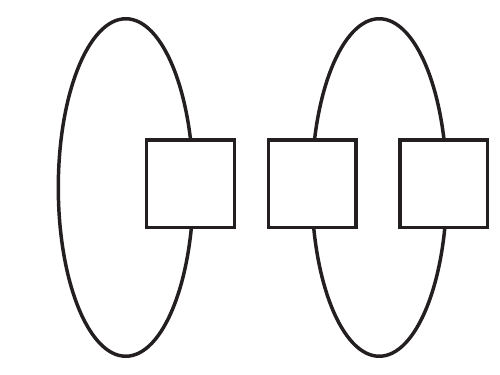}}
\put(-50,0){\includegraphics{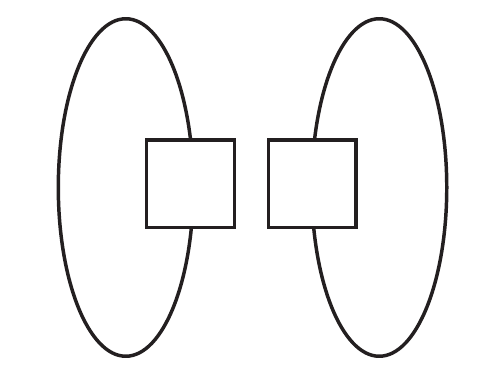}}
\put(100,0){\includegraphics{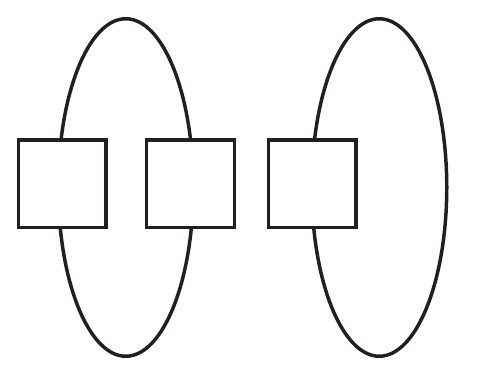}}
\put(92,120){\includegraphics{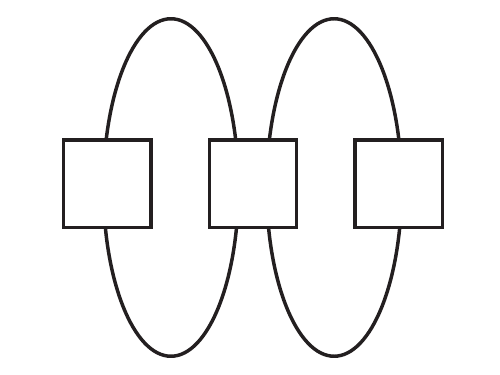}}
\put(-76,173){$J_{11}=$}
\put(-3,173){$L_1$}
\put(31,173){$B$}
\put(70,173){$\mathbb{L}_1$}

\put(242,173){$=J_{21}$}
\put(117,173){$L_2$}
\put(160,173){$\mathbb{B}$}
\put(202,173){$\mathbb{L}_1$}

\put(-76,53){$J_{12}=$}
\put(-1,53){$L_1$}
\put(35,53){$\mathbb{L}_2$}

\put(242,53){$=J_{22}$}
\put(112,53){$L_2$}
\put(150,53){$B$}
\put(186,53){$\mathbb{L}_2$}

\end{picture}
\caption{Derivatives $\frac{\partial K}{\partial \mathfrak{m}_{ij} }$}\label{fig:4derivs}
\end{figure}
Note that each is a boundary link and hence each has maximal Alexander nullity. Hence by Corollary~\ref{cor:firstrho=rhozeroofder} the first-order signatures of $K$ corresponding to these Lagrangians are $\{\rho^f_0(J_{ij})\}$, where $f$, each case, is the abelianization map. Each link $J_{ij}$ is obtained from the trivial link of $2$-components by one or more infections using knots. Thus by Lemma~\ref{lem:additivity} we see that
$$
\rho_0(J_{11})=\rho_0(L_1)+\rho_0(B)+\rho_0(\mathbb{L}_1)>0,
$$
$$
\rho_0(J_{12})=\rho_0(L_1)+\rho_0(\mathbb{L}_2)=0,
$$
$$
\rho_0(J_{21})=\rho_0(L_2)+\rho_0(B)+\rho_0(\mathbb{L}_1)>0,~~\text{and}
$$
$$
\rho_0(J_{22})=\rho_0(L_2)+\rho_0(B)+\rho_0(\mathbb{L}_2)>0.
$$
Since $L_1$ and $\mathbb{L}_2$ are algebraically slice, their zero-th order signatures vanish. The other $\rho_0$ are, by hypothesis, positive or possibly zero in the case of $B$. Hence each of the first-order signatures of $K$ corresponding to Lagrangians is positive, except the one corresponding to $J_{12}$, which is zero. Therefore,  first-order signatures cannot distinguish $K$ from a slice knot. Indeed if $L_1$ and $\mathbb{L}_2$ were chosen to be slice knots then $K$ would be a slice knot. In fact, as we shall show in Example~\ref{ex:mainexampleshighergenus}, no matter what choice is made for $L_1$ and $\mathbb{L}_2$ no metabelian invariants can distinguish $K$ from a slice knot.
\end{ex}

We close this section with the proof of Lemma \ref{lem:lagrangianisrep}.

\begin{proof}[Proof of Lemma \ref{lem:lagrangianisrep}] Suppose $P$ is a Lagrangian of rank $d$ over $\Q$ so the rank of $\mathcal{A}_0(K)$ is $2d$. Let $\Sigma$ be a genus $g$ Seifert surface for $K$. Then necessarily $d\leq g$. If we fix an identification of $S^3-\Sigma$ with a fundamental domain of the infinite cyclic cover of $S^3-K$ then we have maps
$$
\phi: H_1(S^3-\Sigma;\Z)\hookrightarrow H_1(S^3-\Sigma;\Z)\otimes \Q\overset{i_*}\twoheadrightarrow \mathcal{A}_0(K).
$$
Then $(i_*)^{-1}(P)$ is a vector space of dimension $r=2g-d\geq g$ that splits as $\text{ker}(i_*)\oplus V$ where $\text{ker}(i_*)$ has dimension $2g-2d$ and $V$ has dimension $d$. Choose a $\Z$-basis $\{\gamma_1,...\gamma_r,...,\gamma_{2g}\}$ of $H_1(S^3-\Sigma;\Z)$ such that  $\{\gamma_{g+1}\otimes 1,...,\gamma_{g+d}\otimes 1\}$ is a basis of $V$ and
$$
\{\gamma_{d+1}\otimes 1,...,\gamma_{g}\otimes 1\}\cup\{\gamma_{g+d+1}\otimes 1,...,\gamma_{2g}\otimes 1\}
$$
is a basis of $\text{ker}(i_*)$. Thus
\begin{equation}\label{eq:alpha0}
\{\phi(\gamma_{1}),...,\phi(\gamma_{d})\} ~\text{is a basis for} ~~\mathcal{A}_0(K)/P
\end{equation}
and
\begin{equation}\label{eq:alpha1}
\{\phi(\gamma_{g+1}),...,\phi(\gamma_{g+d})\} ~\text{is a basis for} ~~P
\end{equation}
while
\begin{equation}\label{eq:alpha2}
 0=\phi(\gamma_{d+1})=...=\phi(\gamma_{g})
\end{equation}
and
\begin{equation}\label{eq:alpha3}
0=\phi(\gamma_{g+d+1})=...=\phi(\gamma_{2g}).
\end{equation}
Since $P\subset P^\perp$, $\mathcal{B}\ell_0(\phi(\gamma_{g+i}),\phi(\gamma_{g+j}))=0$ for all $1\leq i,j\leq g$. Following ~\cite[proof of Theorem 2]{Ke} and ~\cite[page 122-123]{Hi}, this implies that the Seifert form vanishes on $\{a_1,...,a_g\}$ where $\{a_i|~1\leq i\leq 2g\}$ is the basis of $H_1(\Sigma)$ that is dual to $\{\gamma_i\}$ under ``linking in $S^3$''. Let $\mathfrak{m}$ be the $\Z$-span of $\{a_1,...,a_g\}$, clearly a metabolizer.
Since the Seifert form vanishes on $\{a_1,...,a_g\}$, the intersection form on $H_1(\Sigma)$ also vanishes on $\{a_1,...,a_g\}$. That is, $\mathfrak{m}$ is a maximal isotropic subgroup of $H_1(\Sigma)$ with respect to the intersection form so it can be extended to a symplectic basis $\mathbb{A}=\{a_1,...,a_g,b_{1},\dots,b_{g}\}$ whose first $g$ elements are a basis for $\mathfrak{m}$. Let $\mathcal{A}=\{\alpha_{1},..,\alpha_{g},\beta_1,..,\beta_g\}$ denote the linking-dual basis of $H_1(S^3-\Sigma;\Z)$. We claim:
\begin{itemize}
\item  [1.] $\{\phi(\alpha_{1}),...,\phi(\alpha_{d})\}$ and $\{b_{1},...,b_{d}\}$ are bases of $\mathcal{A}_0(K)/P$.
\item [2.] $\{a_1,...,a_g\}$ and $\{\phi(\beta_{1}),...,\phi(\beta_{g})\}$ span $P$ in the rational vector space $\mathcal{A}_0(K)$.
\end{itemize}
To establish the first claim, let $N=(n_{ij})$ be the matrix that transforms (by left multiplication) a column vector in the $\mathcal{A}$ coordinates into a vector in the $\gamma$ coordinates, so that for $j\leq g$
$$
\alpha_j=\sum_{i=1}^{2g}n_{ij}\gamma_{i},
$$
and for $j> g$
$$
\beta_{j-g}=\sum_{i=1}^{2g}n_{ij}\gamma_{i}.
$$
For any fixed $k\leq g$ take the linking number of each side of these equations with $a_k$. Since $\mathcal{A}$ is dual to $\mathbb{A}$, and $\{\gamma_i\}$ is dual to $\{a_i\}$ this yields
$$
\delta_{kj}=n_{kj}.
$$
Thus $N$ is given by a block matrix
$$
\left(\begin{matrix} I & 0\cr B & C
\end{matrix}\right),
$$
for some invertible $g\times g$ matrix $C$. Therefore for $j\leq g$
\begin{equation}\label{eq:alpha'1}
\alpha_j=\gamma_j+\sum_{i=g+1}^{2g}n_{ij}\gamma_{i};
\end{equation}
while for $j>g$
\begin{equation}\label{eq:alpha'2}
\beta_{j-g}=0+\sum_{i=g+1}^{2g}n_{ij}\gamma_{i}.
\end{equation}
In particular, we see that the expressions for $\{\alpha_{d+1},...\alpha_g\}$ and for $\{\beta_1,..,\beta_g\}$ do not involve $\{\gamma_1,...,\gamma_d\}$. Thus combining Equations ~\ref{eq:alpha0}-\ref{eq:alpha'2} we see that the sets $\{\phi(\alpha_{d+1}),...,\phi(\alpha_g)\}$ and $\{\phi(\beta_1),..,\phi(\beta_g)\}$ lie in $P$. Hence $\{\phi(\alpha_{1}),...,\phi(\alpha_{d})\}$ is a basis of $\mathcal{A}_0(K)/P$, establishing the first half of the first claim. Moreover, just as in the proof of Proposition~\ref{prop:lagrangprops}, $\phi(\alpha_{i})=\pm(t-1)[b_{i}]$ in $\mathcal{A}_0(K)$. Thus the sets $\{b_{1},...,b_{d}\}$ and $\{\phi(\alpha_{1}),...,\phi(\alpha_{d})\}$ generate the same subspace of $\mathcal{A}_0(K)$. This finishes the proof of claim $1$.

Similarly, to establish claim $2$ we need only show that $\{\phi(\beta_{1}),...,\phi(\beta_{g})\}$ spans $P$. We have already remarked that the span of $\{\phi(\beta_{1}),...,\phi(\beta_{g})\}$ is \emph{contained} in $P$. Left multiplication by  $N^{-1}$ transforms $\gamma$-coordinates into $\mathcal{A}$-coordinates. Note that $N^{-1}$ is given by
$$
\left(\begin{matrix} I & 0\cr -C^{-1}B & C^{-1}
\end{matrix}\right).
$$
From the form of this matrix we see that each of $\{\gamma_{g+1},...,\gamma_{2g}\}$ can be written in terms of $\{\beta_{1},...,\beta_{g}\}$. Therefore the span of $\{\phi(\alpha_{g+1}),...,\phi(\alpha_{2g})\}$ contains the span of $\{\phi(\gamma_{g+1}),...,\phi(\gamma_{2g})\}$, which contains $P$ by Equation~\ref{eq:alpha1}. Thus $\{\phi(\beta_{1}),...,\phi(\beta_{g})\}$ spans $P$ as claimed.

In particular, by claim $2$, $\mathfrak{m}$ represents $P$.
\end{proof}

\section{First-order signatures for links}\label{sec:firstordersigslinks}

We want to define the second-order signatures of a knot $K$ to be union of the first-order signatures of its derivatives. If $K$ has genus greater than one, its derivatives are links. Thus we need to define first-order signatures for links. There is very little in the literature about this topic. The reader interested only in genus one knots can skip this section.

We extend our notion of first-order signatures to pairs $(J,f)$ where $J=\{J_1,...,J_m\}$ is an ordered, oriented \emph{link} with trivial linking numbers and $f:\pi_1(M_J)\to A\cong\Z^k$ is an  epimorphism. Let $G=\pi_1(M_J)$. Recall that $\mathcal{A}_0^f(J)$ denotes the $\mathbb{Q}[\mathbb{Z}^k]$-torsion submodule of $H_1(M_J;\mathbb{Q}[\mathbb{Z}^k])$.

\begin{defn}\label{defn:linkfirstordersignatures} A \textbf{first-order $\boldsymbol{L^{(2)}}$-signature of a pair $\boldsymbol{(J,f)}$} is a real number $\rho(M_J, \phi)$ where $\phi:G\to G/K$ for some $\ker(\phi)=K\lhd G$ such that
\begin{itemize}
\item [1)] $G/K$ is PTFA;
\item [2)] $ G^{(2)}_{r}\subset K\subset \ker(f)$ so there is a commutative diagram
\begin{equation}\label{eq:firstordersig}
\begin{CD}
\phi:~G      @>>>    G/G^{(2)}_r  @>>> G/K \\
  @VVidV   @VVpV        @VVpV       \\
f:~G     @>>>  G/G^{(1)}_r    @>>>  G/G^{(1)}_{r}K @>>> \Z^k \\
\end{CD}
\end{equation}
\noindent Since $H_1(M_J;\Z[\Z^k])$ may be interpreted as the module $\ker(f)/[\ker(f),\ker(f)]$, there is a natural map
    $$
   \eta:~ K\subset\ker(f)\to H_1(M_J;\Z[\Z^k])\to H_1(M_J;\Q[\Z^k]),
   $$
   with respect to which we require that
\item [3)]   $\eta(K)\cap \mathcal{A}_0^f(J)$ spans an isotropic submodule of $\mathcal{A}_0^f(J)$, with respect to $\mathcal{B}\ell_0^f(J)$.
   \item [4)] $G^{(1)}/K \otimes_{\Z[\Z^k]} \Q[\Z^k]$  is a finitely-generated torsion $\Q[\Z^k]$-module.
\end{itemize}
\end{defn}

The last property will not play a role in this paper.

\begin{ex}\label{ex:firstsiglink} Consider the link $J$ shown in Figure~\ref{fig:linkfirstsigs}. Suppose $f:\pi_1(M_J)\to \Z^2$ is the abelianization map. We shall show below that any first-order signature of $(J,f)$ is a first-order signature of $J_1$ possibly added to $\rho_0(J_2)$. It will follow that, for example, if we choose $J_2$ to be trivial and choose $J_1$ to be any one of the various knots created in Examples~\ref{ex:knotfirst-ordersigs},~\ref{ex:figeight} or ~\ref{ex:nonzerofirst-ordersigs}, then all of the first-order signatures of the resulting link $J$ are non-zero.
\begin{figure}[htbp]
\setlength{\unitlength}{1pt}
\begin{picture}(200,150)
\put(20,0){\includegraphics{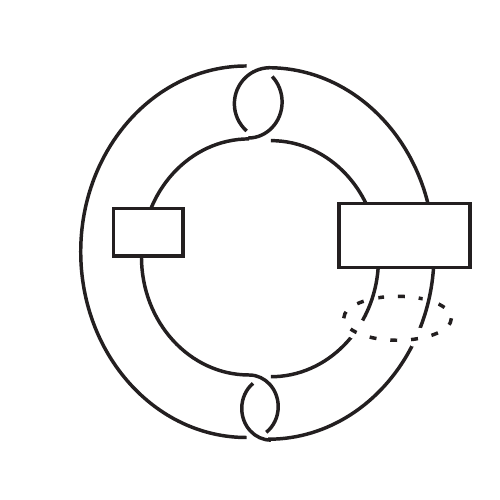}}
\put(0,73){$J~=$}
\put(156,50){$\eta$}
\put(56,74){$J_1$}
\put(130,73){$J_2$}
\end{picture}
\caption{}\label{fig:linkfirstsigs}
\end{figure}

Let $G=\pi_1(M_J)$ and let $\phi:G\to G/K$ correspond to a first-order signature for $(J,f)$. By ~\cite[Lemma 2.3]{CHL3}
$$
\rho(M_J,\phi)=\rho(M_U,\phi_U)+\rho(M_1,\phi_1)+\rho(M_2,\phi_2)
$$
where $U$ is the trivial link of $2$ components, $M_i$ are the zero-framed surgeries on $M_{J_i}$ and the $\phi_i$ are induced maps.
By that same Lemma, since $\eta$ is a commutator, the last term is either $0$ or $\rho_0(J_2)$, according as to whether $\eta\in K$ or not. Moreover it is easy to see that \textbf{any} $\rho$-invariant of $M_U=\#_{i=1}^2 S^1\times S^2$ is zero since any map $\phi$ extends over the exterior of the complement of trivial slice disks ($\natural _{i=1}^2~S^1\times B^3$)~\cite[Theorem 4.2]{COT}. Thus we need only verify that $\rho(M_1,\phi_1)$ is a first-order signature of $J_1$. Note that $\ker(f)=G^{(1)}$ so $\mathcal{A}_0^f(J)$ is merely the torsion submodule of $H_1(M_J;\Q[x^{\pm 1},y^{\pm 1}])$. Consider the commutative diagram below where $H=\pi_1(S^3-J_1)$ and $i:S^3-J_1\hookrightarrow M_J$ is the inclusion.
$$
\begin{diagram}\label{diagramex3}\dgARROWLENGTH=1.2em
\node{\eta_1:\ker(\phi_1)}\arrow{s,r}{i_*}\arrow{e,t}{}\node{H^{(1)}} \arrow{s,r}{i_*}\arrow{e,t}{}\node{H^{(1)}/H^{(2)}} \arrow{e,t}{}
\arrow{s,r}{i_*}\node{\mathcal{A}_0(J_1)}
\arrow{s,r}{i_*}\\
\node{\eta_J:\ker(\phi)=K}\arrow{e,t}{}\node{ G^{(1)}}\arrow{e,t}{}\node{H_1(M_J;\Z[x^{\pm 1},y^{\pm 1}])}\arrow{e,t}{}\node{H_1(M_J;\Q[x^{\pm 1},y^{\pm 1}])}
\end{diagram}
$$
A Mayer-Vietoris sequence as in ~\cite[Theorem 8.2]{C}\cite{Lei3} establishes that
\begin{equation}\label{eq:tensoredup}
H_1(M_J;\Q[x^{\pm 1},y^{\pm 1}])\cong H_1(M_U;\Q[x^{\pm 1},y^{\pm 1}])\oplus \mathcal{A}_0(J_1)\otimes_{\Q[t^{\pm 1}]}\Q[x^{\pm 1},y^{\pm 1}]
\end{equation}
where $t$ maps to the product of meridians $xy$. Since $U$ is trivial, $H_1(M_U;\Q[x^{\pm 1},y^{\pm 1}])$ is torsion-free. Hence $i$ induces an isomorphism
$$
i_*:\mathcal{A}_0(J_1)\otimes_{\Q[t^{\pm 1}]}\Q[x^{\pm 1},y^{\pm 1}]\cong \mathcal{A}_0^f(J).
$$
It follows that
$$
\eta_1(\ker(\phi_1))\subset \eta(K)\cap \mathcal{A}_0^f(J)
$$
which, by part $3)$ of Definition~\ref{defn:linkfirstordersignatures}, is isotropic with respect to the linking form on $\mathcal{A}_0^f(J)$ . This implies that any two elements of $\eta_1(\ker(\phi_1))$ have zero Blanchfield pairing when considered as elements in $\mathcal{A}_0^f(J)$. But the direct sum decomposition of Equation~\ref{eq:tensoredup} extends to the level of Blanchfield forms and since $t=x$ induces an embedding $\Z\hookrightarrow \Z\times\Z$ we may conclude that $\eta_1(\ker(\phi_1))$ is in fact isotropic with respect to the usual linking form on $\mathcal{A}_0(J_1)$ (all of this is detailed in ~\cite[Theorem 3.7]{Lei3}\cite[Theorem 3.3]{CHL4}. Thus $\phi_1$  corresponds to a first-order signature of $J_1$.
\end{ex}

\section{Second Order $L^{(2)}$ Signatures}\label{sec:secondorder sigs}

Second-order signatures for knots are loosely speaking, von Neumann $\rho$ invariants associated to coefficient systems that factor through $G/G^{(3)}$ and that might arise as the coefficient system associated to a slice disk exterior. They are potentially stronger than abelian or metabelian signatures. In this generality these have appeared already in ~\cite[Section 4]{COT}. However, in this generality, there are two serious problems. First, there are an infinite number of such coefficient systems. Secondly, it is not fully known how to use the condition ``might arise as the coefficient system associated to a slice disk exterior'' to restrict this number. Since the theorems are of the nature: ``If $K$ is slice then \emph{one} of the second-order signatures is zero'', these problems make such theorems often useless in practice. Here we make progress towards restricting the number of possible coefficient systems that need be considered, enough so that, especially for genus one knots, it is often finite.


\begin{defn}\label{defn:secondordersignatures}
Given an algebraically slice knot $K$, let $\mathcal{P}$ be the set of Lagrangians $P$ of $\mathcal{A}_0(K)$ for which the corresponding
first-order signature of $K$ is zero (see section 4). For each $P\in\mathcal{P}$, choose a metabolizer $\mathfrak{m}_{P}$ representing $P$. A \textbf{complete set of second-order $\boldsymbol{L^{(2)}}$-signatures} of $K$ is:
$$\bigcup_{P\in\mathcal{P}}\{\text{first-order signatures of }\frac{\partial K}{\partial\mathfrak{m}_{P}}\}.$$
\end{defn}

Note that if the first-order signature of $K$ corresponding to $P$ is non-zero then $P$ cannot be the Lagrangian corresponding to an actual slice disk for $K$. Therefore there is no need to consider second-order signatures for such $P$. If $\Delta_K(t)=1$ then $(\frac{\partial K}{\partial \mathfrak{m} },f)$ has $f=0$ so all the second-order signatures are zero.


It is an important point that a complete set of second-order signatures is often finite, especially for genus one knots. Any algebraically slice genus 1 knot $K$ has two Lagrangians (0 if $\Delta_{K}=1$), so the set $\mathcal{P}$ above is finite. Then, if each derivative $J_i$, $i=1,2$, can be chosen so that $\mathcal{A}_0(J_i)$ is cyclic, then the total number of distinct \emph{submodules} of $\mathcal{A}_0(J_i)$ is finite. Thus certainly the number of first-order signatures of $J_i$ is finite. This is assured, for example, if $J_i$ is a 2-bridge knot or if the Alexander polynomial is not divisible by the square of a prime polynomial (as in Example \ref{ex:derivsofhighergenus}).

A complete set of second-order signatures is not a knot invariant. However we do claim that these sets can obstruct a knot's being a slice knot. For example, the following greatly generalizes the theorem of Gilmer (and Cooper). The proof of these results will be given in Section~\ref{sec:nullbordisms}.

\begin{thm}\label{thm:main} If $K$ is a genus one slice knot then any set of second-order signatures (constructed using a genus one surface) contains zero. Specifically, for any genus one Seifert surface $\Sigma$, there is a homologically essential simple closed curve $J$ of self-linking zero on $\Sigma$, which has vanishing zero-th order signature and a vanishing first-order signature. (Beware that, if $\Delta_K(t)=1$ then, even if $K$ is not slice, the latter signatures of $J$ will be zero by definition since $f$ will be trivial).
\end{thm}

\begin{ex}\label{ex:secordersiggenusone} We now give one of the promised families of knots for which the slice obstructions given by our second-order signatures are stronger than those imposed by any abelian or metabelian invariants but which cannot be detected by the second-order invariants of ~\cite{COT} and cannot be detected by the techniques of ~\cite{CHL4} (since they are not iterated satellites). Consider the arbitrary genus one knot $K$ as shown on the left-hand side of Figure~\ref{fig:arbgenus1AS}. Assume that the knot $L_1$ (shown dashed in Figure~\ref{fig:arbgenus1AS} and discussed in Example~\ref{ex:arbgenusonederiv}) is algebraically slice. Then $K$ cannot be distinquished from a slice link by any metabelian invariants. (In fact, since $L_1$ is $(.5)$-solvable by ~\cite[Remark 1.3]{COT}, $K$ is $(1.5)-$solvable by ~\cite[Thms.8.9, 9.11]{COT}.) Since we make no assumptions on the link $L$, but only on the knot type of its components, there are many examples where $L$ is not a split link (see right-hand side of Figure~\ref{fig:arbgenus1AS}) nor even a satellite.

Now suppose $t=0$ for simplicity, and assume $\ell\notin \{-1,0\}$ to ensure that $\Delta(K)\neq 1$. Then, as we saw in Example~\ref{ex:arbgenusonederiv}, $K$ has two Lagrangians, $P_i$, $i=1,2$, with metabolizers $\mathfrak{m}_i$ represented by the cores of the two bands and derivatives $\frac{\partial K}{\partial \mathfrak{m}_i }$ equal to the knot types $L_i$
of the components of $L$. By Corollary~\ref{cor:firstrho=rhozeroofder} the first-order signature of $K$ corresponding to $P_i$ is $\rho_0(L_i)$ (we use that $\Delta_K\neq 1$ to ensure that $f\neq 0$). Suppose that $\rho_0(L_2)\neq 0$.
Then by Definition~\ref{defn:secondordersignatures}, the set of second-order signatures of $K$ is the set of first-order signatures of the knot $L_1$. Finally choose $L_1$ to be an algebraically slice knot all of whose first-order signatures are non-zero, such as one of the families given in Examples~\ref{ex:knotfirst-ordersigs} and ~\ref{ex:nonzerofirst-ordersigs}. Then this complete set of second-order signatures does not contain zero. Therefore any such $K$ violates Theorem~\ref{thm:main} and thus is not a slice knot.
\end{ex}

If the genus of $K$ is greater than one then its derivatives will be links. Here, for simplicity, our results are restricted to links $J$ with maximal ``higher-order Alexander nullities''.

\begin{thm}\label{thm:main2} Suppose $K$ is a slice knot with the property that for each Lagrangian $P$ for which the first-order signature of $K$ corresponding to $P$ vanishes, it is possible to choose a representative $(\frac{\partial K}{\partial \mathfrak{m}},f)$ that is a link of maximal Alexander nullity. Then some member of any complete set of second-order signatures (computed using \textbf{such} representatives) has absolute value at most genus$(\Sigma)~-1$. Moreover, if it is possible to choose each representative $(J,f)$ to be an infected trivial link then any complete set of second-order signatures (computed using \textbf{such} representatives) contains zero.
\end{thm}

Here $\Sigma$ is the Seifert surface used to compute $\frac{\partial K}{\partial \mathfrak{m}}$. Note that in the case of genus one Seifert surfaces $\frac{\partial K}{\partial \mathfrak{m}}$ is a knot, which always has maximal Alexander nullity. Thus Theorem~\ref{thm:main} is a special case of Theorem~\ref{thm:main2}.

\begin{ex}\label{ex:mainexampleshighergenus} We now exhibit families of higher-genus non-satellite knots with vanishing classical and metabelian invariants for which the second-order signatures of Theorem~\ref{thm:main2} obstruct their being slice knots. Moreover these examples cannot be detected by techniques of ~\cite{COT} (they are even distinct up to concordance). Since they are not formed by iterated satellite constructions, the techniques of ~\cite{CHL3} cannot be directly applied.
Let $K$ be one of the knots of Figure~\ref{fig:derivsgenustwo} under the assumptions of Example~\ref{ex:derivsofhighergenus}. In that example we computed that $K$ has $4$ Lagrangians, only one of which has a zero first-order signature. In this case $\frac{\partial K}{\partial \mathfrak{m}}$ is the split link $\{L_1,\mathbb{L}_2\}$. Since these knots are algebraically slice, $K$ is $(1.5)$-solvable by ~\cite[Theorem 8.9]{COT}. Hence it cannot be distinquished from slice knot by any metabelian invariants ~\cite[Theorems 9.11,9.12]{COT}.

For simplicity assume furthermore that $\mathbb{L}_2$ is the trivial knot. Then, by Definition~\ref{defn:secondordersignatures}, a complete set of second-order signatures of $K$ is the set of first-order signatures of the link $\{L_1,U\}$. It can be checked that the associated map $f$ is the just the abelianization map since the rank of $\mathcal{A}_0(K)/P_{12}$ is $2$. This set of first-order signatures was computed in Example~\ref{ex:firstsiglink} and does not contain zero as long as $L_1$ is chosen to be an algebraically slice knot that itself has no non-zero first-order signatures,
such as the knots in Examples~\ref{ex:knotfirst-ordersigs} and ~\ref{ex:nonzerofirst-ordersigs}. Thus this complete set of second-order signatures does not contain zero. But $\{L_1,U\}$ is an infected trivial link so $K$ violates the last clause of Theorem~\ref{thm:main2}. Therefore no such $K$ is a slice knot.
\end{ex}

\section{Null-Bordisms}\label{sec:nullbordisms}

Knot concordance is intimately related to $4$-manifolds and to the homology cobordism type of the $3$-manifold obtained by zero surgery on the knot. For example the following are well-known. If $K$ is concordant to $J$ then $M_K$ and $M_J$ are homology cobordant. Moreover $K$ bounds a slice disk in some homology $4$-ball if and only if $M_K$ bounds a $4$-manifold $V$ (namely the exterior of the slice disk) with the homology of $S^1$. In this section we make the geometric observation that there is a canonical cobordism, $E$, between the zero-framed surgeries $M_K$ and $M_{\frac{\partial K}{\partial \mathfrak{m} }}$. If this cobordism were a \emph{homology} cobordism then it would be reasonable to expect that $K$ is slice if and only if one of its derivatives is slice. But this cobordism is \emph{not} a homology cobordism~- it is not even a product on $H_1$. So then the (somewhat surprising) key algebraic step (which already appeared as a tool in ~\cite{CHL1A}\cite{CHL3}) is to show that, nonetheless, such cobordisms are sufficient to guarantee that higher-order signature invariants of $K$ are related to lower-order signature invariants of $\frac{\partial K}{\partial \mathfrak{m} }$. Moreover, by gluing $V$ to $E$ along $M_K$ we will answer the question: ``If $K$ is a slice knot then what type of $4$-manifold does $M_{\frac{\partial K}{\partial \mathfrak{m} }}$ bound?'' We find that this yields a new and useful category of $4$-manifolds, called \textbf{null-bordisms}. We prove that a knot's being null-bordant implies that both its zero-th and first-order signatures vanish. Using this we prove Theorems~\ref{thm:main} and ~\ref{thm:main2}.

\subsection{A cobordism between $M_K$ and $M_{\frac{\partial K}{\partial \mathfrak{m} }}$}\label{subsec:bordisms}\

\

Suppose $K$ is an algebraically slice knot, $\Sigma$ is a genus $g$ Seifert surface for $K$, $\mathfrak{m}$ is a metabolizer for the Seifert form on $H_1(\Sigma)$, and $\frac{\partial K}{\partial \mathfrak{m} }=J=\{J_1,...,J_g\}$. We describe a cobordism, denoted $E$, from $M_K$ to $M_{J}$. These cobordisms are closely related to, but much more general than, the cobordisms in ~\cite[Section 2]{CHL1A}\cite[Section 2]{CHL3}. Let $C$ denote the $4$-manifold obtained from $M_K\times [0,1]$ by adding $2$-handles $\{h_1,...,h_g\}$ along the components of $J$ in $M_K\times \{1\}$ using framing zero with respect to $S^3$. Note that since $J$ forms half of a symplectic basis for $\Sigma$, the latter is surgered in a canonical way to a disk inside $\partial^+C$. Together with a disk bounding (the longitude of) $K$ in $M_K$, this forms a canonical $2$-sphere $S$ embedded in $\partial^+C$. Let $E$ denote the $4$-manifold obtained from $C$ by adding a $3$-handle along $S$.

Here is another point of view from which we see that in fact $\partial^+C\cong M_{J}~\# ~S^1\times S^2$. By definition, $\partial^+C$ is the result of zero-framed Dehn surgery on the components of $J$ in $M_K$. However, since $K$ and $J$ are disjoint and have zero linking numbers in $S^3$, we can reverse the order of the surgeries and consider $\partial^+C$ as the result of a single zero-framed surgery along $K$ viewed as a knot in $M_{J}$. But as observed in the previous paragraph, $K$ is unknotted in $M_{J}$. It follows immediately that $\partial^+C\cong M_{J}~\#~ S^1\times S^2$. It also follows that $\partial^+C$ is the $3$-manifold obtained as the $+$-boundary of $M_{J}\times [0,1]$ after adding a trivial zero-framed $2$-handle.  From this point of view, the subsequent $3$-handle addition precisely cancels this $2$-handle, yielding that $\partial^+E\cong M_{J}$.

Therefore we have established the bulk of:

\begin{prop}\label{prop:Efacts} The following hold for $E$ above
\begin{itemize}
\item [1)] $\partial E=\partial^-E\coprod \partial^+E\cong -M_K\coprod M_{J}$.
\item [2)] The map $i_*:\pi_1(M_K)\to \pi_1(E)$ is surjective with kernel the normal closure of the set of loops represented by the components of $J$.
\item [3)] The meridian of the band on which $J_i\hookrightarrow \Sigma\hookrightarrow M_K=\partial^-E$ lies is isotopic in $E$ to a meridian of $J_i$ in $M_{J}=\partial^+E$.
\item [4)] $H_1(M_K;\mathbb{Z})\to H_1(E;\mathbb{Z})$ is an isomorphism, while
\item [5)] $H_1(M_{J};\mathbb{Z})\to H_1(E;\mathbb{Z})$ is the zero map.
\item [6)] $H_2(M_K;\mathbb{Z})\to H_2(E;\mathbb{Z})$ is the zero map, and
\item [7)] $H_2(E;\mathbb{Z})/i_*(H_2(\partial^+E;\mathbb{Z}))=0$.
\end{itemize}
\end{prop}
\begin{proof} Property $1$ was established above. Property $2$ is immediate from the handle structure of $E$. Since the components of $J$ lie on $\Sigma$, they are null-homologous in $M_K$ and Property $4$ thus follows from Property $2$.

To establish the remaining properties, we consider the following more concrete version of the analysis of $\partial^+C$. Consider zero surgery on the union of $K$ and $J$.
\begin{figure}[htbp]
\setlength{\unitlength}{1pt}
\begin{picture}(207,151)
\put(0,0){\includegraphics{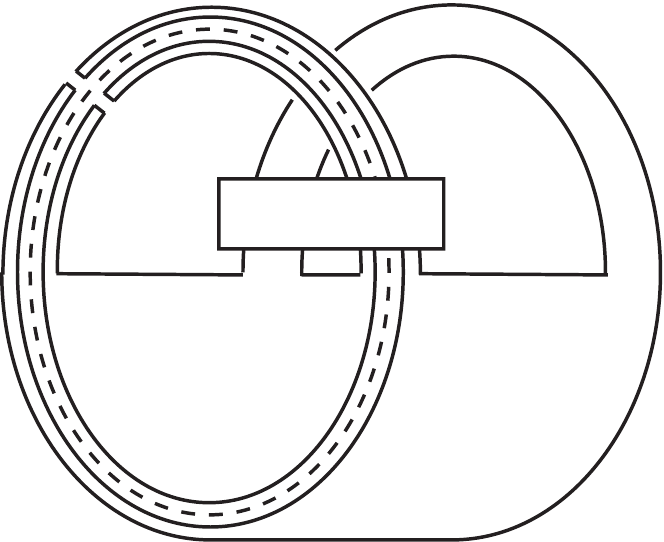}}
\put(89,91){$L$}
\end{picture}
\caption{Image of $K$ after sliding twice over each component of $J$}\label{fig:handleslides}
\end{figure}
For each $i$ slide both strands (of $K$) of the band on which $J_i$ lies over $J_i$. A genus one example is shown in Figure~\ref{fig:handleslides}. Observe that after sliding $K$ $2g$ times in this way, it becomes unknotted and split off from $J$. Thus $\partial^+C\cong M_{J}~\#~ S^1\times S^2$.  Note further that, under this homeomorphism, the meridian, $\alpha_i$, of the band on which $J_i$ lies is carried to the meridian of $J_i$. The addition of the $3$-handle does not alter this fact. This establishes Property $3$. Now since the meridians of the bands of $\Sigma$ are null-homologous in $M_K$, the meridians of the copy of $J$ in $\partial^+E$ are null-homologous in $E$. This establishes Property $5$.

Since $C$ is obtained by adding $2$-handles $\{h_1,...,h_g\}$ along the null-homologous circles $J_i$,  $H_2(C)$ is generated by $H_2(M_K)$ together with $g$ embedded surfaces $F_i$ which can be constructed as follows: choose Seifert surfaces in $S^3$ for the $J_i$ whose interiors avoid $K\cup J$ (here we use that $J$ represents a metabolizer) and then cap these off with copies of the cores of the $2$-handles that lie in $\partial^+C$. By construction,  the $F_i$ lie in $\partial^+C$ so are in the image of $H_2(\partial^+C) \to H_2(C)$.
The final generator (for the image of $H_2(M_K) \to H_2(C)$ can be taken to be a copy of $\Sigma$ (capped off). Note that this also can be taken to lie in  $\partial^+C$. But in $\partial^+C$ we have seen that this capped-off $\Sigma$ is homologous (by surgering along the disks) to the embedded $2$-sphere to which the $3$-handle is attached. This gives Property $6$. We have thus shown that
$$
H_2(\partial^+C)\to H_2(C)\to H_2(E)\cong \mathbb{Z}^g
$$
is surjective with basis $\{F_i\}$. But now if $E-C$ denotes the cobordism from $\partial^+C$ to $\partial^ +E$ consisting of the single $3$-handle addition then
$$
H_2(\partial^+E)\to H_2(E-C)
$$
is surjective since $E-C$ is obtained from $\partial^+ E$ by adding a $1$-handle. Hence
the classes represented by the $F_i$ in $H_2(\partial^+C)\to H_2(E-C)$ have representatives in $\partial^+ E$. This gives Property $7$.
\end{proof}

\subsection{Null-bordism}\label{subsec:nullbordism}\

\

If $K$ is a slice knot then $M_K$ bounds a $4$-manifold with the homology of a circle. In this case, what is true of the zero surgery on a derivative of $K$?

Suppose $K$ is a slice knot, $\Delta$ is a slice disk for $K$, and $V=B^4-\Delta$ so $\partial V=M_K$. Recall from ~\ref{eq:defP} that there is a Lagrangian $P_\Delta$ associated to $\Delta$. Let $J$ be a derivative of $K$ with respect a metabolizer that represents $P_\Delta$ for some Seifert surface, $\Sigma$. Finally let $W$ be the $4$-manifold obtained by gluing $V$ along $M_K$ to the manifold $E$ of Proposition~\ref{prop:Efacts}, so that $\partial W=M_J$. Then by Mayer-Vietoris and Proposition~\ref{prop:Efacts} we easily deduce:

\begin{lem}\label{lem:Wfacts} If $K$ is a slice knot, then corresponding to any slice disk $\Delta$ and any derivative $J$ representing the associated Lagrangian $P_\Delta$ there is a compact oriented $4$-manifold $W$ such that
\begin{itemize}
\item [1)] $\partial W= M_{J}$.
\item [2)] $W=(B^4-\Delta)\cup_{M_K}E$ and the unbased meridian of the $i^{th}$ component of $J$ in $M_{J}$ is isotopic in $E$ to the meridian of the band on which $J_i\hookrightarrow \Sigma\hookrightarrow M_K$ lies;
\item [3)] $H_1(M_K;\mathbb{Z})\to H_1(W;\mathbb{Z})$ is an isomorphism, while
\item [4)] $H_1(M_J;\mathbb{Z})\to H_1(W;\mathbb{Z})$ is the zero map, and
\item [5)] $H_2(W;\mathbb{Z})/i_*(H_2(\partial W;\mathbb{Z}))=0$.
\end{itemize}
\end{lem}

We cull the important properties of the $4$-manifold $W$ into the following definition.

\begin{defn}
\label{defn:nullbordism} A compact, connected oriented topological 4-manifold $W$ with $\partial W = M$ is a \textbf{ null-bordism} for $M$  if
$$
H_2(W;\mathbb{Z})/i_*(H_2(\partial W;\mathbb{Z}))=0;
$$
and, for some integer $m\geq 0$, the map
$$
\pi_1(M)\to \pi_1(W)\to \pi_1(W)/\pi_1(W)^{(m+1)}_r
$$
is not the zero map. Moreover, if $n\geq 0$ is the \emph{minimal} integer such that $M$ that is null-bordant via $W$ then we say that $M$ is \textbf{null-bordant via $W$ at level $n$}. To any such null-bordism, there is associated a non-trivial epimorphism $f:\pi_1(M)\to A\cong\Z^k$ defined as follows. It follows from minimality that $j_*(\pi_1(M))\subset\pi_1(W)^{(n)}_r$ and $j_*(\pi_1(M))\nsubseteq\pi_1(W)^{(n+1)}_r$. Consequently the restriction
$$
\phi:\pi_1(M)\to \pi_1(W)/\pi_1(W)^{(n+1)}_r
$$
is non-trivial and factors through the abelianization of $\pi_1(M)$. Let $f:\pi_1(M)\to \text{image}(\phi)\cong\Z^k$ be the induced abelian representation (considered only up to post-composition with an isomorphism unless $H_1(M)$ has a natural basis). Then we say that \textbf{$(M,f)$ is null-bordant via $W$ at level $n$}. If $M=M_J$ then we say that $(J,f)$ is \textbf{null-bordant} via $W$.
\end{defn}

 For example, recall that if a link $J$ is a slice link with slice disk $\Delta$ then $M_J=\partial W$ where $W=B^4-\Delta$ and $H_2(W)=0$ and $H_1(M_J)\cong H_1(W)$ so $M_J$ is null-bordant via $W$ at level $m=0$. In this case the associated map $f$ is  just the abelianization.

However, the example that motivated Definition~\ref{defn:nullbordism} is really that provided by Lemma~\ref{lem:Wfacts}. Let us formalize this.

\begin{prop}\label{prop:derivisnullbord} Suppose that $K$ is a slice knot whose Alexander polynomial is not $1$, $P_\Delta$ is the Lagrangian associated to a slice disk $\Delta$ as in ~\ref{eq:defP}, and $(J,f)=\frac{\partial K}{\partial\mathfrak{m}}$ where $\mathfrak{m}$ represents $P_\Delta$. Then $(M_J,f)$ is null-bordant at level $1$ (via the $W$ constructed in Lemma~\ref{lem:Wfacts}).
\end{prop}
\begin{proof} Applying properties $1$ and $5$ of Lemma~\ref{lem:Wfacts}, we see that we need only show that the composition
$$
\phi:\pi_1(M_J)\overset{j_*}{\to} \pi_1(W)\to \pi_1(W)/\pi_1(W)^{(2)}_r
$$
is not the zero map, and that $\tilde{f}\equiv\phi:\pi_1(M_J)\to \text{image}(\phi)$ is in fact identifiable to $f$. By property $2$ of Lemma~\ref{lem:Wfacts}, the meridians of the components of $J$ in $\pi_1(M_J)=\partial W$ are freely homotopic to the meridians, $\alpha_i$, of the bands of a Seifert surface for $K$ in $M_K= \partial(B^4-\Delta)\subset W$. Thus $\pi_1(M_J)$ maps under inclusion into the commutator subgroup $\pi_1(W)^{(1)}$ and hence $\phi$ and $\tilde{f}$ factor through the abelianization $H_1(M_J)\cong \mathbb{Z}^{g}$ and the image of $\tilde{f}$ is the subgroup generated by the images of $\{\alpha_1,...,\alpha_g\}$. Therefore $\tilde{f}$ is equivalent to the map
$$
\langle\alpha_1,...,\alpha_g\rangle\subset \pi_1(M_K)^{(1)} \overset{j_*}{\to} \pi_1(V)^{(1)}/\pi_1(V)^{(2)}_r\to \pi_1(W)^{(1)}/\pi_1(W)^{(2)}_r
$$
since $M_K\hookrightarrow W$ factors through $V$.
Now consider the following commutative diagram, which we now proceed to justify.
$$
\begin{diagram}\label{diagram:f}\dgARROWLENGTH=1.2em
\node{\{a_1,...,a_g\}\cup\{\alpha_1,...,\alpha_g\}} \arrow{e,t}{j_{\ast}}
\arrow{s,r}{} \arrow{s,l}{} \node{\pi_1(V)^{(1)}/\pi_1(V)^{(2)}_r} \arrow{e,t}{\cong}
\arrow{s,r}{i}\node{\pi_1(W)^{(1)}/\pi_1(W)^{(2)}_r}\\
\node{\mathcal{A}_0(K)}\arrow{e,t}{j_*}\node{H_1(V;\Q[t,t^{-1}])}\arrow{e,t}{\cong}\node{\pi_1(V)^{(1)}/\pi_1(V)^{(2)}_r\otimes \Q}
\end{diagram}
$$
 Recall that by construction $W$ is obtained from $V$ by adding 2-handles along the curves $\{a_1,...,a_g\}$, which form a basis of $\mathfrak{m}$ which in turn spans $P$. Since by Equation~(\ref{eq:defP}), $P$ is the kernel of the map $j_*$ on the bottom row of the diagram, the elements $\{j_*(a_1),...,j_*(a_g)\}$ are zero in $H_1(V;\Q[t,t^{-1}])$. But also recall that
\begin{equation}\label{eq:alex}
H_1(V;\Q[t,t^{-1}])=H_1(V;\Z[\pi_1(V)/\pi_1(V)^{(1)}])\otimes \Q[t,t^{-1}]\cong \pi_1(V)^{(1)}/\pi_1(V)^{(2)}_r\otimes \Q.
\end{equation}
Since
\begin{equation}\label{eq:inj}
\pi_1(V)^{(1)}/\pi_1(V)^{(2)}_r\hookrightarrow \pi_1(V)^{(1)}/\pi_1(V)^{(2)}_r\otimes \Q
\end{equation}
is injective, ~\ref{eq:alex} and \ref{eq:inj} combine to show that $\{j_*(a_1),...,j_*(a_g)\}\subset\pi_1(V)^{(2)}_r$ (where this $j_*$ is from the top row of the diagram). Thus
the inclusion $V\hookrightarrow W$ induces an isomorphism
$$
\pi_1(V)/\pi_1(V)^{(2)}_r\cong \pi_1(W)/\pi_1(W)^{(2)}_r
$$
and hence (since $H_1(W)\cong H_1(V)\cong \Z$ by property $3$ of Lemma~\ref{lem:Wfacts}) an isomorphism
$$
\pi_1(V)^{(1)}/\pi_1(V)^{(2)}_r\cong \pi_1(W)^{(1)}/\pi_1(W)^{(2)}_r.
$$
Therefore $\tilde{f}$, up to isomorphism, is equivalent to the map
$$
\langle\alpha_1,...,\alpha_g\rangle\subset \pi_1(M_K)^{(1)} \overset{j_*}{\to} \pi_1(V)^{(1)}/\pi_1(V)^{(2)}_r.
$$
So we need only understand the subgroup of
$$
\pi_1(V)^{(1)}/\pi_1(V)^{(2)}_r
$$
that is spanned by $\{j_*(\alpha_1),...,j_*(\alpha_g)\}$. Moreover by equations ~\ref{eq:alex} and \ref{eq:inj} it suffices to consider these curves in $H_1(V;\Q[t,t^{-1}]))$. Since
$$
\mathcal{A}_0(K)/P\cong \text{image}(j_*)\subset H_1(V;\Q[t,t^{-1}])
$$
up to isomorphism it suffices to consider the subgroup of $\mathcal{A}_0(K)/P$ generated by $\{\alpha_1,...,\alpha_g\}$.

On the other hand, recall that the definition of the map $f$ associated to $\frac{\partial K}{\partial\mathfrak{m}}$ is given by the composition
$$
\pi_1(M_J)\twoheadrightarrow H_1(M_J)\cong \langle\alpha_1,...,\alpha_g\rangle\to \mathcal{A}_0(K)/P.
$$
Hence $\tilde{f}$ and $f$ are identical up to isomorphism. Moreover, by Proposition~\ref{prop:lagrangprops}, $\{\alpha_1,...,\alpha_g\}$ spans $\mathcal{A}_0(K)/P$.  Since the Alexander polynomial is not $1$, $\mathcal{A}_0(K)/P$ is non-trivial.  Thus $\tilde{f}$ is non-trivial.
This concludes the verification that $W$ is a null-bordism for $(M_J,f)$ at level $1$.
\end{proof}

\subsection{Zero-th order signatures vanish for null-bordant knots and links}\

\

We will show that all null-bordant knots and, under some restrictions, null-bordant links have vanishing zero-th order signatures. More precisely, if $(J,f)$ is null-bordant via $W$ then there corresponds a particular zero-th order signature defined as $\rho^f_0(J)=\rho(M_J,f)$. We claim:

\begin{thm}\label{thm:nbimpliesrho=0} If the $c$ component link $(J,f)$ is null-bordant via $W$ then
$$
|\rho^f_0(J)|\leq c-1-\eta(J,f).
$$
\end{thm}

For a knot, there is only one non-trivial zero-th order signature so $\rho^f_0(J)=\rho_0(J)$. Moreover $\eta(J,f)=0=c-1$ by ~\cite[Lemma 2.11]{COT}. Thus Theorem~\ref{thm:nbimpliesrho=0} takes a very simple form for knots.

\begin{cor}\label{cor:nbimpliesrho=0} If $J$ is a knot that is null-bordant  then $\rho_0(J)=0$.
\end{cor}

\textbf{Open Problem}: If $J$ is a knot that is null-bordant, is $J$ necessarily of finite order in Levine's algebraic concordance group? We remark that, since the figure-eight knot is slice in a $\Q$-homology $4$-ball $W$ with $H_1(W)\cong\Z$, one sees that it is null-bordant, yet not zero in Levine's group.

\begin{proof}[Proof of Theorem~\ref{thm:nbimpliesrho=0}] Let $\G=\pi_1(W)/\pi_1(W)^{(n+1)}_r$ ($n$ minimal) and let $\tilde{\phi}$ be the canonical quotient map which is non-trivial by hypothesis. We claim that $\rho(M_J,\tilde{\phi})=0$. We deduce this from the following special case of a previous result of the authors.

\begin{thm}\label{thm:CHLsigvanishes}(Cochran-Harvey-Leidy)~\cite[Theorem 5.9, Remark 5.11]{CHL3} Suppose $W$ is a  null-bordism and
$\tilde{\phi}:\pi_1(W)\ra\G$ is a coefficient system where
$\G$ is a PTFA group. If the restriction of $\tilde{\phi}$ to each component of $\partial W$ is non-trivial on $\pi_1$ then
$$
|\rho(\partial W,\tilde{\phi})|\leq \beta_1(\partial W)-\beta_0(W)-\text{rank}_{\mathbb{Z}\G}H_1(\partial W;\mathbb{Z}\G) .
$$
\end{thm}

\noindent For us $\partial W=M_J$. Since
$$
\Z^d=\text{image}(\tilde{\phi}_{|\pi_1(M_J)})\hookrightarrow \G
$$
is a monomorphism, $\mathbb{Z}\G$ is a free, hence a flat, $\Z[\Z^d]$-module. Thus
$$
H_1(M_J;\mathbb{Z}\G)\cong H_1(M_J;\Z[\Z^d])\otimes_{\Z[\Z^d]}\Z\G.
$$
The same fact holds for the respective quotient fields. Thus
\begin{equation}\label{eq:ranksame}
\text{rank}_{\mathbb{Z}\G}H_1(M_J;\mathbb{Z}\G)=\text{rank}_{\mathbb{Z}[\Z^d]}H_1(M_J;\phi)=\eta(M_J,f).
\end{equation}
Therefore we can apply Theorem~\ref{thm:CHLsigvanishes} to conclude that
$$
|\rho(M_J,\tilde{\phi})|\leq \beta_1(M_J)-1-\eta(J,f).
$$
Since the image of $\tilde{\phi}$ is abelian, $\rho(M_J,\tilde{\phi})$ is the zero-th order signature that we have denoted $\rho^f_0(J)$. This concludes the proof of Theorem~\ref{thm:nbimpliesrho=0}.
\end{proof}

\begin{proof}[Proof of Proposition~\ref{prop:firstrho=rhozeroofder}] If $\Delta_K=1$ then $P=0$ and $f=0$. Then the result is trivially true since both $\rho$-invariants are zero. Suppose $\Delta_K\neq 1$ and that $P$ is represented by $(J,f)=\frac{\partial K}{\partial \mathfrak{m} }$, a link of $c$ components. Consider the cobordism $E$ from $M_K$ to $M_J$ given by Proposition~\ref{prop:Efacts}. Let $G=\pi_1(M_K)$ and $\phi:G\to G/G^{(2)}P$ be the coefficient system corresponding to $P$. By Definition~\ref{defn:firstordersignatures}, the first-order signature of $K$ corresponding to $P$ is $\rho(M_K,\phi)$. Since the components of $J$ span $\mathfrak{m}$ which represents $P$, the components of $J\hookrightarrow S^3-K$ represent elements of $P$ by Definition~\ref{def:represents}. Thus, by $2$ of Proposition~\ref{prop:Efacts}, $\phi$ extends to $\pi_1(E)$. We claim that the restriction, $\tilde{\phi}$, to $\pi_1(M_J)$ of this extended $\phi$ is merely $f$ followed by an embedding. This was essentially already verified in the proof of Proposition~\ref{prop:derivisnullbord} (replacing $W$ by $E$ and ignoring $V$). Therefore
$$
\partial(E,\phi)=(M_K,\phi)~ \amalg~(-M_J,i\circ f).
$$
We may now apply Theorem~\ref{thm:CHLsigvanishes} to $(E,\phi)$ to conclude that
$$
|\rho(M_K,\phi)-\rho_0^f(J)|\leq \beta_1(M_J)-1-\eta(J,f)-\text{rank}H_1(M_K;\phi).
$$
Here we have used the same argument as for ~\ref{eq:ranksame} above to equate $\eta(J,f)$ with $\text{rank}H_1(M_J;\phi)$. Moreover, since $\beta_1(M_K)=1$, $\text{rank}_{\mathbb{Z}\G}H_1(M_K;\mathbb{Z}\G)=0$ for any non-trivial coefficient system by ~\cite[Lemma 2.11]{COT}. Thus
$$
|\rho(M_K,\phi)-\rho_0^f(J)|\leq c-1-\eta(J,f),
$$
as claimed.
\end{proof}

\subsection{First-order signatures vanish for null-bordant knots}\

\

We will also show that all null-bordant knots have vanishing first-order signatures. This can be made more precise. Recall that to each isotropic submodule, $P\subset\mathcal{A}_0(J)$ there corresponds a first-order signature, $\rho(M_J,\phi_P)$. We claim that each null-bordism $W$ induces a particular such isotropic submodule, $P_W$ (just like a slice disk exterior) .

\begin{lem}\label{lem:bordisminducesP} If $J$ is a knot that is null-bordant via $W$ then the inclusion $M_J\hookrightarrow W$ induces a an isotropic submodule, $P_W\subset\mathcal{A}_0(J)$.
\end{lem}

\begin{proof}[Proof of Lemma~\ref{lem:bordisminducesP}] Let $n\geq 0$ be the (minimal) integer such that $M_J$ that is null-bordant via $W$ at level $n$. Consider the coefficient system $\tilde{\psi}:\pi\to \pi/\pi^{(n+1)}_r\equiv\Lambda$ whose restriction to $\pi_1(M_J)$ we call $\psi$. Since $\psi$ factors non-trivially through $\Z$,
$$
H_1(M_J;\mathbb{Q}\Lambda)\cong H_1(M_J;\mathbb{Q}[t,t^{-1}]) \otimes_{\mathbb{Q}[t,t^{-1}]}\mathbb{Q}\Lambda\equiv  \mathcal{A}_0(J) \otimes_{\mathbb{Q}[t,t^{-1}]}\mathbb{Q}\Lambda.
$$
We now invoke a special case of a previous theorem of the authors.

\begin{thm}\label{thm:nontriviality}(Cochran-Harvey-Leidy)~\cite[Theorem 6.6]{CHL3} Suppose $W$ is a null-bordism for $M_J$ ($J$ a knot), $\Lambda$ is a PTFA group and $\tilde{\psi}:\pi_1(W)\to \Lambda$ is a coefficient system whose restriction to $\pi_1(M_J)$ factors non-trivially through $\mathbb{Z}$. Let $P$ be the kernel of the composition
$$
\mathcal{A}_0(J)\overset{id\otimes 1}\lra  \mathcal{A}_0(J) \otimes_{\mathbb{Q}[t,t^{-1}]}\mathbb{Q}\Lambda \overset{\cong}{\to} H_1(M_J;\mathbb{Q}\Lambda)\overset{j_*}\to H_1(W;\mathbb{Q}\Lambda).
$$
Then $P\subset P^\perp$ with respect to $\mathcal{B}l_0(J)$, the classical Blanchfield linking form on $\mathcal{A}_0(J)$.
\end{thm}
Setting $P_W=P$ we are done.
\end{proof}

Now we can state:

\begin{thm}\label{thm:nbimpliesfirstsig=0} If $J$ is a null-bordant knot then one of the first-order signatures of $J$ is zero. Specifically, if $J$ is null-bordant via $W$ then $\rho(M_J,\phi_{P_{W}})=0$.
\end{thm}

\textbf{Open Problem}: Suppose $J$ is a null-bordant knot. What can be said about its Casson-Gordon invariants?

The first-order signatures of null-bordant \emph{links} are highly constrained, but in this case nullities and higher-order nullities enter into the picture. We discuss a result only in some simpler cases where these nullities are maximal. It is also true that any null-bordism for a link $(J,f)$ corresponds to a particular first-order signature of $(J,f)$. For simplicity we will not state and prove this here. Rather, the interested reader will see that its verification is part of our proof of Theorem~\ref{thm:nblinkimpliesfirstsig=0}.

\begin{thm}\label{thm:nblinkimpliesfirstsig=0} If $(J,f)$ is a link of $c$ components with maximal Alexander nullity, i.e. $\eta(J,f)=c-1$, that is null-bordant via $W$ with $\beta_1(W)=1$, then one of the first-order signatures of $(J,f)$ is at most $c-1$ in absolute value.  Moreover if  $(J,f)$ is an infected trivial link then one of the first-order signatures of $(J,f)$ is $0$. In each case, the vanishing signature is the one corresponding to the null-bordism $W$.
\end{thm}

\begin{proof}[Theorem~\ref{thm:nblinkimpliesfirstsig=0} implies Theorem~\ref{thm:nbimpliesfirstsig=0}] If $J$ is a null-bordant knot then the associated epimorphism $f$ is merely the abelianization. Thus the Alexander nullity $\eta(J,f)$ is zero. Since $c=1$ the result follows from Theorem~\ref{thm:nblinkimpliesfirstsig=0}.
\end{proof}

Before proving Theorems~\ref{thm:nbimpliesfirstsig=0} and ~\ref{thm:nblinkimpliesfirstsig=0}, we show how they imply our main theorems, Theorem~\ref{thm:main} and Theorem~\ref{thm:main2}.

\begin{proof}[Proof of Theorem~\ref{thm:main2}] If $\Delta_K(t)=1$ the the theorem is true since all the signatures are zero. Suppose $K$ is a slice knot and $\Delta_K(t)\neq 1$. Then the $\mathbb{Q}$-rank of $\mathcal{A}_0(K)$ is greater than zero. By ~\ref{eq:defP}, any slice disk, $\Delta$, corresponds to a particular Lagrangian $P_\Delta$. By Theorem~\ref{thm:firstordersigs=0} the corresponding first-order signature of $K$ vanishes. By hypothesis there is a representative of $P_\Delta$, $J=\frac{\partial K}{\partial \mathfrak{m}}$, wherein $(J,f)$ has maximal Alexander nullity (or, in the second case, $J$ is an infected trivial link). Assume that we have chosen such a representative and let $\Sigma$ denote the chosen Seifert surface. By Proposition~\ref{prop:derivisnullbord}, $(J,f)$ is null-bordant at level $1$. By Theorem~\ref{thm:nblinkimpliesfirstsig=0}, one of the first-order signatures of $(J,f)$ is at most genus$(\Sigma$)$-1$ in absolute value (or, in the second case, is zero). Then, by definition, any set of second-order signatures for $K$ that is constructed using, for each Lagrangian, such a representative, contains a number of absolute value at most genus$(\Sigma$)$-1$ (or, in the second case, contains zero).
\end{proof}

Even though Theorem~\ref{thm:nbimpliesfirstsig=0} is a special case of Theorem~\ref{thm:nblinkimpliesfirstsig=0}, we give an independent proof of it for the sake of clarity.

\begin{proof}[Proof of Theorem~\ref{thm:nbimpliesfirstsig=0}] Let $n\geq 0$ be the minimal integer such that $M_J$ that is null-bordant via some $W$ at level $n$. Let $\G=\pi_1(W)/\pi_1(W)^{(n+2)}_r$, let $\tilde{\phi}:\pi_1(W)\to \G$ be the canonical quotient map and let $\phi$ denote the restriction of $\tilde{\phi}$ to $\pi_1(M_J)$. Note that $\phi$ is non-trivial since, by hypothesis, the composition
$$
\psi:\pi_1(M_J)\overset{\phi}{\longrightarrow}\pi_1(W)/\pi_1(W)^{(n+2)}_r\to \pi_1(W)/\pi_1(W)^{(n+1)}_r
$$
is non-trivial. Since $\beta_1(M_J)=1$, $\text{rank}_{\mathbb{Z}\G}H_1(M_J;\mathbb{Z}\G)=0$ by ~\cite[Lemma 2.11]{COT}. Therefore we may apply Theorem~\ref{thm:CHLsigvanishes} to $W$ and $\tilde{\phi}$ and conclude that
$$
\rho(M_J,\phi)= 0.
$$

It only remains to identify $\rho(M_J,\phi)$ as a first-order signature of $J$. Since $n$ was chosen to be minimal, $j_*(\pi_1(M_J))\subset\pi_1(W)^{(n)}_r$ and consequently $\phi(\pi_1(M_J)^{(2)})=0$. Thus $\phi$ factors as
$$
\pi_1(M_J)\twoheadrightarrow \pi_1(M_J)/\pi_1(M_J)^{(2)}\to \text{image}(\phi).
$$
where the image of $\phi$ is metabelian. Let $\phi$ (continue to) denote this epimorphism with restricted range. For simplicity let $\pi=\pi_1(W)$ and $G=\pi_1(M_J)$. Since $j_*(G)\subset\pi^{(n)}_r$, $\phi(G^{(1)})\subset\pi_{r}^{(n+1)}$. Thus $G^{(1)}\subset \ker{\psi}$ and so the image of $\psi$ is a non-trivial abelian subgroup of $\pi^{(n)}_r/\pi^{(n+1)}_r$. Since the latter group is torsion-free abelian, the image of $\psi$ is infinite cyclic, generated by the meridian, $\mu$, of $J$. We claim that the kernel of $\phi:G\to \G$ is contained in $G^{(1)}$. For suppose $x \in \ker{\phi}$ and $x=\mu^{m}y$
 where $y \in G^{(1)}$. Since $x \in \ker{\phi}$, clearly $x\in \ker{\psi}$. Since $y\in \ker{\psi}$, $\mu^{m}\in\ker{\psi}$, but this contradicts the fact that $\psi(\mu)$ has infinite order unless $m=0$. Thus $\ker{\phi}\subset G^{(1)}$.

It remains to show that $\ker{\phi}= \ker{(G^{(1)}\to G^{(1)}/G^{(2)}\to \mathcal{A}_0(J)/P)}$ for some submodule $P\subset \mathcal{A}_0(J)$ such that $P\subset P^\perp$ with respect to the classical Blanchfield form on $J$. Consider the coefficient system $\tilde{\psi}:\pi\to \pi/\pi^{(n+1)}_r\equiv\Lambda$ whose restriction to $\pi_1(M_J)$ we have called $\psi$. Since $\psi$ factors through $\Z$,
$$
H_1(M_J;\mathbb{Q}\Lambda)\cong H_1(M_J;\mathbb{Q}[t,t^{-1}]) \otimes_{\mathbb{Q}[t,t^{-1}]}\mathbb{Q}\Lambda.
$$
Consider the following commutative diagram where $i$ is injective.
\begin{equation*}
\begin{CD}
\pi_1(M_J)^{(1)}      @>\equiv>>    \pi_1(M_J)^{(1)}  @>j_*>> \pi^{(n+1)}_r @>>>
\pi^{(n+1)}_r/\pi^{(n+2)}_r \\
  @VVpV   @VVV        @VVV       @VViV\\
\mathcal{A}_0(J)     @>id\otimes 1>>  H_1(M_J;\mathbb{Q}\Lambda)    @>j_*>> H_1(W;\mathbb{Q}\Lambda) @>\cong>>
  (\pi^{(n+1)}_r/[\pi^{(n+1)}_r,\pi^{(n+1)}_r])\otimes_\mathbb{Z} \mathbb{Q}\\
\end{CD}
\end{equation*}
The composition in the top row is $\phi$. By definition (Lemma~\ref{lem:bordisminducesP}), the kernel of the composition in the bottom row is $P_W$, which, by Theorem~\ref{thm:nontriviality}, is an isotropic submodule. Since $i$ is injective, it follows that $\ker{\phi}=p^{-1}(P_W)$. It follows that $P_W$ is the desired submodule $P$ of the Alexander module referred to above.

This completes the verification that $\rho(M_J,\phi)$ is a first-order signature of $J$ and hence completes the proof of Theorem~\ref{thm:nbimpliesfirstsig=0}.

\end{proof}

\begin{proof}[Proof of Theorem~\ref{thm:nblinkimpliesfirstsig=0}] Suppose $(M_J,f)$ that is null-bordant via $W$ at level $n$. Let $\G=\pi_1(W)/\pi_1(W)^{(n+2)}_r$, let $\tilde{\phi}:\pi_1(W)\to \G$ be the canonical quotient map, and let $\phi$ denote the restriction of $\tilde{\phi}$ to $\pi_1(M_J)$. Note that $\phi$ is non-trivial since, by definition of null-bordism, the composition
$$
\psi:\pi_1(M_J)\overset{\phi}{\longrightarrow}\pi_1(W)/\pi_1(W)^{(n+2)}_r\overset{p}{\longrightarrow} \pi_1(W)/\pi_1(W)^{(n+1)}_r
$$
is the non-trivial $f:\pi_1(M_J)\to A\cong\Z^d$ followed by an embedding. Therefore we may apply Theorem~\ref{thm:CHLsigvanishes} to $W$ and $\tilde{\phi}$ and conclude that
$$
|\rho(M_J,\phi)|\leq \beta_1(M_J)-1-\text{rank}_{\Z\G}H_1(M_J;\Z\G)\leq c-1,
$$
where $J$ has $c$ components. In the special case that $J$ is an infected trivial link, observe that any longitude $\ell$ of an infecting knot $K$ lies in
$$
\pi_1(S^3-K)^{(2)}\subset \pi_1(M_J)^{(2)} \subset \ker\phi.
$$
It then follows directly from ~\cite[Lemma 6.8]{CHL4} that
$$
\text{rank}_{\Z\G}H_1(M_J;\Z\G)=\beta_1(M_J)-1,
$$
so $\rho(M_J,\phi)=0$.

It only remains to identify $\rho(M_J,\phi)$ as a first-order signature of $(J,f)$. Let $\pi=\pi_1(W)$ and $G=\pi_1(M_J)$. Since $n$ is minimal, $j_*(G)\subset\pi^{(n)}_r$ and consequently $j_*(G^{(2)})=0$. Let $K$ denote the kernel of $\phi$. Thus $\phi$ factors as
$$
G\twoheadrightarrow G/G^{(2)}_r\to \text{image}(\phi)=G/K.
$$
 Note that $G^{(2)}_r\subset K$ and $G/K$ is PTFA since it is a subgroup of the PTFA group $\G$. Then consider the commutative diagram

\begin{equation}\label{eq:firstordersig2}
\begin{CD}
G       @>>> G/K @>>>{\hookrightarrow} \pi/\pi^{(n+2)}_r\\
  @VVidV           @VVV      @VVpV \\
G        @>>>  \text{image}(f) @>i>> \pi/\pi^{(n+1)}_r\\
\end{CD}
\end{equation}
By definition of null-bordism, the bottom composition is $f$ followed by an embedding $i$. Thus $K\subset \ker(f)$. This establishes properties
$1)$ and $2$ of Definition~\ref{defn:linkfirstordersignatures}.

Consider the coefficient system $\tilde{\psi}:\pi\to \pi/\pi^{(n+1)}_r\equiv\Lambda$ whose restriction to $\pi_1(M_J)$ we denote $\psi$. Since $\psi=i\circ f$ where $i$ is injective
\begin{equation}\label{eq:flat}
H_1(M_J;\mathbb{Q}\Lambda)\cong H_1(M_J;\mathbb{Q}[\Z^d]) \otimes_{\mathbb{Q}[\Z^d]}\mathbb{Q}\Lambda\cong\big(\ker(f)/[\ker(f),\ker(f)]\big)\otimes_{\mathbb{Z}[\Z^d]}\mathbb{Q}\Lambda.
\end{equation}
Consider the following commutative diagram.

\begin{equation*}
\begin{CD}
K\cap \eta^{-1}(\mathcal{A}_0^f(J))      @>\subset>>    \ker(f)  @>j_*>> \pi^{(n+1)}_r @>>>
\pi^{(n+1)}_r/\pi^{(n+2)}_r \\
  @VV{\eta}V   @VVV        @VVV       @VViV \\
\mathcal{A}_0^f(J)     @>id\otimes 1>>  H_1(M_J;\mathbb{Q}\Lambda)    @>j_*>> H_1(W;\mathbb{Q}\Lambda) @>\cong>>
  \frac{\pi^{(n+1)}_r}{[\pi^{(n+1)}_r,\pi^{(n+1)}_r]}\otimes_\mathbb{Z} \mathbb{Q}\\
\end{CD}
\end{equation*}
The composition in the top row is a restriction of $\phi$ and is identically zero since $K=\ker(\phi)$. Let $P'$ be the kernel of the composition in the bottom row. It follows that if $P$ is the span of $\eta(K)\cap \mathcal{A}_0^f(J)$ then $P\subset P'$. We now need only show that $P$ is an isotropic submodule with respect to the ordinary Blanchfield form on $\mathcal{A}_0^f(J)$ since this will  complete the verification of property $3$ of Definition~\ref{defn:linkfirstordersignatures}. For this we need a special case of a previous theorem of the authors.

\begin{thm}\label{thm:selfannihil}(Cochran-Harvey-Leidy)~\cite[Theorem 6.3]{CHL3} Suppose $M_J$ is null-bordant via $W$ and $\tilde{\psi}:\pi_1(W)\ra\Lambda$ is a non-trivial coefficient system where
$\Lambda$ is a PTFA group. Suppose that $\text{rank}_{\mathbb{Q}\Lambda}H_1(M_J;\mathbb{Q}\Lambda)=\beta_1(M_J)-1$. If $P$ is the kernel of the inclusion-induced map
$$
TH_1(M_J;\mathbb{Q}\Lambda)\xrightarrow{j_{\ast}} TH_1(W;\mathbb{Q}\Lambda),
$$
then $P\subset P^\perp$ with respect to the Blanchfield form on $TH_1(M_J;\mathbb{Q}\Lambda)$.
\end{thm}
By (~\ref{eq:flat}) above,
$$
\text{rank}_{\mathbb{Q}\Lambda}H_1(M_J;\mathbb{Q}\Lambda)= \text{rank}_{\mathbb{Q}[\Z^d]}H_1(M_J;\mathbb{Q}[\Z^d])=\text{rank}\mathcal{A}_o^f(J)=\eta(J,f),
$$
which equals $\beta_1(M_J)-1$ by hypothesis. Thus the hypotheses of Theorem~\ref{thm:selfannihil} are satisfied.  Now suppose $x,y\in P\subset P'$. Then $\{x\otimes 1,y\otimes 1\}\subset \tilde{P}$. Apply Theorem~\ref{thm:selfannihil} to conclude that
$$
\mathcal{B}l^{M_J}_{\mathbb{Q}\Lambda}(x\otimes 1,y\otimes 1)=0.
$$
By the arguments of ~\cite[Section 6]{CHL4},
$$
\mathcal{B}l^{M_J}_{\mathbb{Q}\Lambda}(x\otimes 1,y\otimes 1)=\ov\phi(\mathcal{B}l_0^f(x,y))=0
$$
where $\ov\phi$ is the map on quotient fields induced by the embedding $\Z^d\hookrightarrow \Lambda$ and $\mathcal{B}l_0^f$ is the Blanchfield form on $\mathcal{A}_0^f(J)$. By the argument of ~\cite[Lemma 6.5]{CHL3} $\ov\phi$ is injective. Thus $\mathcal{B}l_0^f(x,y)=0$. Hence $P$ is isotropic with respect to $\mathcal{B}l_0^f$.

To verify property $4$ of Definition~\ref{defn:linkfirstordersignatures}, note that  $H_1(W;\mathbb{Q}\Lambda)$ is a $\Q[\Z^d]$-module via the embedding $i$ of Diagram~\ref{eq:firstordersig2}. Since $\beta_1(W)=1$ and $W$ is compact, this is a finitely-generated torsion module ~\cite[Lemma 2.10]{COT}. Since $G^{(1)}/K$ embeds in $\pi^{(n+1)}_r/\pi^{(n+2)}_r$ by Diagram~\ref{eq:firstordersig2},
$$
G^{(1)}/K \otimes_{\Z[Z^d]} \Q[Z^d]\subset H_1(W;\mathbb{Q}\Lambda).
$$
Thus $G^{(1)}/K \otimes_{\Z[Z^d]} \Q[Z^d]$ is a torsion module.

This completes the verification that $\rho(M_J,\phi)$ is a first-order signature of $J$ and hence completes the proof of Theorem~\ref{thm:nblinkimpliesfirstsig=0}.

\end{proof}

\section{Antiderivatives of Links}\label{sec:antiderivatives}

Suppose $J=\{J_1,...,J_g\}$ is a link in $S^3$. Suppose $V$ is a $2g\times 2g$ Seifert matrix for some algebraically slice knot with respect to a symplectic basis whose first $g$ elements generate a metabolizer. We describe a simple procedure to create a knot $K$, called \textbf{an antiderivative of J} which possesses a Seifert surface $\Sigma$ and symplectic basis that realizes $V$ as its Seifert matrix (so its first $g$ elements generate a metabolizer $\mathfrak{m}$) such that $\frac{\partial K}{\partial \mathfrak{m} }=J$. To form $\Sigma$, and hence $K$, start with $g$ zero-twisted annuli (called $a$-bands) whose cores form the components of $J$. Then add a band (called the $i^{th}$  b-band) to the $i^{th}$ a-band, fusing the inner boundary circle of the a-band to its outer boundary circle. One has tremendous freedom in choosing these b-bands. Choose the twisting of the b-bands and the linking between the b-bands to mimic $V$. The result is the disjoint union of $g$ punctured tori. Now band these together using more bands to arrive at a genus $g$ connected surface $\Sigma$. The boundary of $\Sigma$ is the desired knot $K$, which has the required properties by construction. If the extra data of an epimorphism $f:J\to \Z^d$ is given then the antiderivative of $(J,f)$, denoted $\int (J,f)$ can be defined similarly so that $\frac{\partial K}{\partial \mathfrak{m} }=(J,f)$, except that the map $f$ restricts what Seifert matrices can be realized. In particular the Alexander polynomial of the antiderivative of $(J,f)$ must have degree $2d$. For example if $f$ is the zero map then any $\int (J,f)$ must have trivial Alexander module. Note that in this case one instance of $\int (J,f)$ is obtained by choosing the b-bands as simple as possible in which case the constructed antiderivative is the unknot.

 Suppose a fixed Seifert form has two ``independent'' metabolizers, that is $\mathfrak{m}_1\cap \mathfrak{m}_2=0$ as represented by a matrix $V$ with two disjoint $g\times g$ blocks of zeros. Then if two $g$-component links $J_1$, $J_2$ are given, one can modify the above procedure to choose the cores of the b-bands to form the link $J_2$ and in this way construct an antiderivative $K$ realizing $V$ and such that:
 $$
 \frac{\partial K}{\partial \mathfrak{m}_i }=J_i~~i=1,2.
 $$
 However if the metabolizers are not independent, it seems that such a result should not be expected.

\section{Extension of Results: the $(n)$-solvable filtration}\label{sec:nsolvable}

We explain how our results can be extended to show that first and second-order signatures obstruct a knots lying in certain terms of the $(n)$-solvable filtration ($n\in \frac{1}{2}\Z$) of ~\cite[Section 7,8]{COT}. The notion of $(n)$-solvability and the notion of null-bordism (Definition~\ref{defn:nullbordism}) have a common generalization called \textbf{null-($n$)-bordism} that was introduced in ~\cite[Section 5]{CHL3}. Although we shall not state the full extensions of our results to this category, this notion does arise in some of the proofs below. Recall that for $M=\partial V$ to be an $(n)$-solution one requires that $H_1(M;\Z)\to H_1(V;\Z)$ be an isomorphism, whereas for a null-bordism there is no such requirement. However, for a null-bordism $H_2(V)/H_2(\partial V)=0$, whereas for an $(n)$-solution, $H_2(V)$ is allowed but is of a special type. For a null-($n$)-bordism, we impose no condition on $H_1$ but require that $H_2(V;\Z)/H_2(\partial V)$ have special representatives just as in the definition of an $(n)$-solution. We also say \textbf{$V$ is a null-($n$)-bordism for $\partial V$ at level $m$} if, in addition,
$$
\pi_1(\partial V)\to \pi_1(V)/\pi_1(V)^{(m+1)}_r
$$
is non-trivial and $m$ is minimal for this property.

Recall that associated to any slice disk $\Delta$ for a knot $K$ was a Lagrangian $P_\Delta\subset \mathcal{A}_0(K)$ (see ~\ref{eq:defP}). This was derived from considering the inclusion $M_K\hookrightarrow B^4-\Delta$. This generalizes in an identical fashion to any $(n)$-solution $V$ for $K$ ($V$ replacing $B^4-\Delta$) as long as $n\geq 1$. So if $K\in \mathcal{F}_{(n)}$ where $n\geq 1$ and $V$ is an $(n)$-solution for $K$, then there is a \textbf{corresponding Lagrangian} $P_V$  (by ~\cite[Theorem 4.4, $n=1$]{COT}). If $V$ is merely a null-($n$)-bordism at some level then there is merely an associated \emph{isotropic submodule} which may not be a Lagrangian (just as in Lemma~\ref{lem:bordisminducesP}).

First we state the generalizations of first-order signatures to obstructions to $(1.5)$-solvability. Theorem~\ref{thm:strongerfirstordersigs=0} and Theorem~\ref{thm:strongerCooperThm} were shown previously by Cochran-Orr-Teichner, but are stated here for completeness. Corollary~\ref{cor:strongersliceimpliesrhobound} is new.

\begin{thm}[Generalization of Theorem~\ref{thm:firstordersigs=0};~{\cite[Thms 4.2, 4.4]{COT}~\cite[Proposition 5.8]{CHL3}}]\label{thm:strongerfirstordersigs=0} If $K\in \mathcal{F}_{(1.5)}$ then, for any Lagrangian $P_V$ that corresponds to a $(1.5)$-solution $V$, the corresponding first-order $L^{(2)}$-signature of $K$ vanishes. Thus if $K\in \mathcal{F}_{(1.5)}$  then the set of all first-order signatures corresponding to Lagrangians contains $0$.
\end{thm}

Combining Proposition~\ref{prop:firstrho=rhozeroofder} with Theorem~\ref{thm:strongerfirstordersigs=0} we get a generalization of Corollary~\ref{cor:sliceimpliesrhobound}.
\begin{cor}\label{cor:strongersliceimpliesrhobound} If $K\in \mathcal{F}_{(1.5)}$, $P$ is a Lagrangian corresponding to a $(1.5)$-solution and $(J,f)=\frac{\partial K}{\partial \mathfrak{m} }$ is a $c$-component link where $\mathfrak{m}$ represents $P$, then
$$
|\rho^f_0(J)|\leq c-1-\eta(J,f).
$$
\end{cor}

A specific case of this, when $c=1$, yields a generalization of Cooper's Theorem, which is due to Cochran-Orr-Teichner.

\begin{thm}[Generalization of Theorem~\ref{thm:CooperThm} (Cochran-Orr-Teichner {\cite[Thm. 5.2 ]{COT2}})]\label{thm:strongerCooperThm} If $K\in \mathcal{F}_{(1.5)}$ is a genus one knot then, for any genus one Seifert surface $\Sigma$, there is a homologically essential simple closed curve of self-linking zero on $\Sigma$ which has vanishing zero-th order signature. (Beware that if $\Delta_K(t)=1$ then the latter signature is zero by definition).
\end{thm}

We can generalize our main results to show that second-order signatures obstruct a knot lying in $ \mathcal{F}_{(2.5)}$.

\begin{thm}[Generalization of Theorem~\ref{thm:main}]\label{thm:strongermain}  If $K\in \mathcal{F}_{(2.5)}$ is a genus one knot, then for any genus one Seifert surface $\Sigma$, there is a homologically essential simple closed curve $J$ of self-linking zero on $\Sigma$ which has vanishing zero-th order signature and a vanishing first-order signature. (Beware that if $\Delta_K(t)=1$ then the latter signatures are zero by definition).
\end{thm}

Theorem~\ref{thm:strongermain} is a special case of the following theorem.

\begin{thm}[Generalization of Theorem~\ref{thm:main2}]\label{thm:strongermain2}  Suppose $K\in \mathcal{F}_{(2.5)}$ with the property that for each Lagrangian $P$ for which the first-order signature of $K$ corresponding to $P$ vanishes, it is possible to choose a representative $(\frac{\partial K}{\partial \mathfrak{m}},f)$ that is a link of maximal Alexander nullity. Then some member of any complete set of second-order signatures (computed using \textbf{such} representatives) has absolute value at most genus$(\Sigma)~-1$. Moreover, if it is possible to choose each representative $(J,f)$ to be an infected trivial link then any complete set of second-order signatures (computed using \textbf{such} representatives) contains zero.
\end{thm}

\begin{proof}[Proof of Theorem~\ref{thm:strongermain2}] Suppose $M_K$ is $(2.5)$-solvable via $V$. Then
\begin{equation}\label{eq:defPV}
P_V=\text{ker}\left(\mathcal{A}_0(K)\equiv H_1(M_K;\Q[t,t^{-1}])\to H_1(V;\Q[t,t^{-1}])\right).
\end{equation}
is a Lagrangian by ~\cite[Theorem 4.4, $n=1$]{COT}. By hypothesis, there is a Seifert surface, $\Sigma_{P_V}$, a metabolizer $\mathfrak{m}$ representing $P_V$ and a representative $(J,f)=(\frac{\partial K}{\partial \mathfrak{m}},f)$ with $\eta(J,f)=$genus$(\Sigma_{P_V})-1$. Adjoin to $V$ the cobordism $E$ from $M_K$ to $M_J$ as described in Subsection~\ref{subsec:bordisms}. Let $W=V\cup E$ and let
$$
\phi:\pi_1(W)\to \pi_1(W)/\pi_1(W)^{(3)}_r\equiv \G
$$
be the projection. Then, since $V$ is a $(2.5)$-solution, by ~\cite[Theorem 4.2, $n=2$]{COT},
$$
\rho(M_K,\phi)=0.
$$
Now note that, since $\Delta_K\neq 1$, the restriction of $\phi$ to $\pi_1(M_J)$ is nontrivial by $(3)$ of Proposition~\ref{prop:Efacts}. It follows that $E$ is a null-bordism (with two boundary components here).  Hence Theorem~\ref{thm:CHLsigvanishes} may be applied to $(E,\phi)$ to yield
$$
|\rho(M_J,\phi)-\rho(M_K,\phi)|\leq \beta_1(M_K)+\beta_1(M_J)-2-\text{rank}_{\mathbb{Z}\G}H_1(M_J;\mathbb{Z}\G)-\text{rank}_{\mathbb{Z}\G}H_1(M_K;\mathbb{Z}\G) .
$$
Since $\beta_1(M_K)=1$ the latter rank term vanishes by ~\cite[Lemma 2.11]{COT}. Substituting what we know yields
$$
|\rho(M_J,\phi)|\leq \text{genus}(\Sigma_P)-1-\text{rank}_{\mathbb{Z}\G}H_1(M_J;\mathbb{Z}\G)\leq \text{genus}(\Sigma_P)-1.
$$
In the special case that $J$ is an infected trivial link it follows directly from ~\cite[Lemma 6.8]{CHL4} that
$$
\text{rank}_{\Z\G}H_1(M_J;\Z\G)=\beta_1(M_J)-1=\text{genus}(\Sigma_P)-1,
$$
so $\rho(M_J,\phi)=0$.

It only remains to show that $\rho(M_J,\phi)$ is a first-order signature of $J$, under the hypothesis that $\eta(J,f)=\beta_1(M_J)-1$. Even though $W$ is no longer a null-bordism for $(J,f)$, it is a $(2.5)$-bordism (at level $1$) in the sense of ~\cite[Section 5]{CHL3} (the analysis of Proposition~\ref{prop:derivisnullbord} applies to this larger category). Note that $\phi$ restricted to $G=\pi_1(M_J)$ factors  through $G^{(2)}$. Now the proof is identical to the the proof of this fact in the proof of Theorem~\ref{thm:nblinkimpliesfirstsig=0} with $n=1$ (or, in the simpler genus one case, Theorem~\ref{thm:nbimpliesfirstsig=0}). We need only note that the crucial result Theorem~\ref{thm:selfannihil} (or, in the genus one case, Theorem~\ref{thm:nontriviality}) was actually proved in more generality in ~\cite[Theorem 6.6]{CHL3} than that stated here and in fact applies to the $(2)$-bordism $W$.
\end{proof}

\begin{prop}[Generalization of Proposition~\ref{prop:derivisnullbord}]\label{prop:strongerderivisnullbord} Suppose that $K$ ($\Delta_K(t)\neq 1)$ is a knot that is null-$(n)$-bordant, $n\geq 1$, via $V$ at level $m$ wherein the induced isotropic submodule $P_V$ is a Lagrangian. Let $(J,f)=\frac{\partial K}{\partial\mathfrak{m}}$ where $\mathfrak{m}$ represents $P_V$. Then $(M_J,f)$ is null-$(n)$-bordant at level $m+1$ via $V\cup E$ where $E$ is the cobordism from $M_K$ to $M_J$ constructed in Subsection~\ref{subsec:bordisms}.
\end{prop}
\begin{proof}[Proof of Proposition~\ref{prop:strongerderivisnullbord}] Let $E$ be the cobordism from $M_K$ to $M_J$ and set $W=V\cup E$. This mimics the construction of Lemma~\ref{lem:Wfacts} but here $V$ generalizes $B^4-\Delta$. The properties listed in Lemma~\ref{lem:Wfacts} follow easily, except that $H_2(W)/i_*(H_2(\partial W))\cong H_2(V)$. Thus it follows trivially from the definitions that $W$ is an $(n)$-bordism whose boundary is $M_J$. To show that it is at level $m+1$, we to show that
$$
\psi:\pi_1(M_J)\overset{i_*}{\to} \pi_1(W)\to \pi_1(W)/\pi_1(W)^{(m+2)}_r
$$
is not the zero map, and is in fact the map $f$ followed by an embedding. But by hypothesis
$$
\pi_1(M_K)\subset \pi_1(V)^{(m)}_r
$$
so
$$
\psi(\pi_1(M_J))\subset \pi_1(W)^{(m+1)}_r
$$
using $(3)$ and $(5)$ of Proposition~\ref{prop:Efacts}. Hence $\psi$ factors through the abelianization of $\pi_1(M_J)$. The proof now closely follows the proof of Proposition~\ref{prop:derivisnullbord}. Recall that by property $2$ of Lemma~\ref{lem:Wfacts}, the meridians of $M_J$ are freely homotopic in $E$ to the $\alpha_i$. Thus it suffices to understand the (free abelian) subgroup generated by $\alpha_i$ in $\pi_1(W)^{(m+1)}_r/\pi_1(W)^{(m+2)}_r$ under inclusion. Since $W$ is obtained from $V$ by adding $2$-handles along the components of $J$ (and then a 3-handle), assuming these components lie in $\pi^{(m+2)}_r$, then $V\hookrightarrow W$ induces an isomorphism
$$
\pi_1(V)/\pi_1(V)^{(m+2)}_r\cong \pi_1(W)/\pi_1(W)^{(m+2)}_r,
$$
But recall that, by definition, $P_V$ is the kernel of the composition in the bottom row of the diagram:
\begin{equation*}
\begin{CD}
\pi_1(M_K)^{(1)}      @>\equiv>>    \pi_1(M_K)^{(1)}  @>j_*>> \pi^{(m+1)}_r @>>>
\pi^{(m+1)}_r/\pi^{(n+2)}_r \\
  @VVpV   @VVV        @VVV       @VViV\\
\mathcal{A}_0(K)     @>id\otimes 1>>  H_1(M_K;\mathbb{Q}\Lambda)    @>j_*>> H_1(V;\mathbb{Q}\Lambda) @>\cong>>
  (\pi^{(m+1)}_r/[\pi^{(m+1)}_r,\pi^{(m+1)}_r])\otimes_\mathbb{Z} \mathbb{Q}\\
\end{CD}
\end{equation*}
where $\Lambda\equiv\pi_1(V)/\pi_1(V)^{(m+1)}_r$. This diagram establishes that the kernel of the
top row is precisely $p^{-1}(P_V)$. Thus since $J$ represents $P_V$, the components of $J$ are in the kernel of the top composition, hence map into $\pi^{(m+2)}_r$. Since $\pi_1(M_K)\subset \pi_1(V)^{(m)}_r$ and since the $\alpha_i\in \pi_1(V)^{(m+1)}_r$ the problem reduces to knowing the kernel of
$$
\pi_1(M_K)^{(1)}/\pi_1(M_K)^{(2)}\hookrightarrow \pi_1(V)^{(m+1)}_r/\pi_1(V)^{(m+2)}_r,
$$
which is $p^{-1}(P_V)$. The details are in the proof of Proposition~\ref{prop:derivisnullbord}.

Then $(M_J,f)$ is null-$(n)$-bordant at level $m+1$ via $W$.
\end{proof}

\begin{prop}\label{prop:derivativerho1} Suppose $K$ is an algebraically slice genus $1$ knot with $\rho^1(K)\neq 0$ (see Definition~\ref{defn:rho1}) that is null-bordant via $W$ at level $m$. Then (for any genus $1$ Seifert surface) there is a metabolizer $\mathfrak{m}$ representing $P_W$ such that $\frac{\partial K}{\partial \mathfrak{m} }$ is null-bordant at level $m+1$.
\end{prop}
\begin{proof}[Proof of Proposition~\ref{prop:derivativerho1}] Suppose $K$ is null-bordant via $V$ at level $m$ ($m$ minimal). Consider
$$
\tilde{\psi}:\pi_1(V)\to \pi_1(V)/\pi_1(V)_r^{(m+1)}\equiv \Lambda
$$
whose restriction $\psi$ to $\pi_1(M_K)$ factors non-trivially through $\Z$ by hypothesis. By the first paragraph of the proof of Theorem~\ref{thm:nbimpliesfirstsig=0}, $\rho(M_K,\psi)=0$. Moreover, applying Theorem~\ref{thm:nontriviality}, we deduce that the kernel of
$$
\mathcal{A}_0(K)\overset{id\otimes 1}\lra  \mathcal{A}_0(K) \otimes_{\mathbb{Q}[t,t^{-1}]}\mathbb{Q}\Lambda \overset{i_*}{\to} H_1(M_K;\mathbb{Q}\Lambda)\overset{j_*}\to H_1(V;\mathbb{Q}\Lambda)
$$
is an isotropic submodule, $P$. Since $K$ has genus $1$, either $P=0$ or $P$ has $\Q$-rank $1$ and so is a Lagrangian. We claim that the latter holds. To see this, as before, consider the following commutative diagram, letting $\pi=\pi_1(V)$.
\begin{equation*}
\begin{CD}
\pi_1(M_K)^{(1)}      @>\equiv>>    \pi_1(M_K)^{(1)}  @>j_*>> \pi^{(m+1)}_r @>>>
\pi^{(m+1)}_r/\pi^{(m+2)}_r \\
  @VVpV   @VVV        @VVV       @VViV\\
\mathcal{A}_0(K)     @>id\otimes 1>>  H_1(M_K;\mathbb{Q}\Lambda)    @>j_*>> H_1(V;\mathbb{Q}\Lambda) @>\cong>>
  (\pi^{(m+1)}_r/[\pi^{(m+1)}_r,\pi^{(m+1)}_r])\otimes_\mathbb{Z} \mathbb{Q}\\
\end{CD}
\end{equation*}
As we have seen, this diagram establishes that the kernel of
$\psi$ (top row) is precisely $p^{-1}(P)$. If $P$ were trivial then the kernel of $\psi$ would be $\pi_1(M_K)^{(2)}$. In this case we would have $0=\rho(M_K,\psi)=\rho^1(K)$ which by hypothesis is non-zero. This contradiction shows that $P$ is in fact a Lagrangian. Fixing a genus $1$ Seifert surface $\Sigma$, this Lagrangian is represented by a metabolizer $J=\frac{\partial K}{\partial \mathfrak{m} }$. Hence $J$ represents an element of $\pi^{(m+2)}_r$ under the inclusion since it lies in the kernel of the top row of the diagram. By contrast, the meridian, $\alpha_1$, of the band on which $J$ lies, spans $\mathcal{A}_0(K)/P$ by Proposition~\ref{prop:lagrangprops}. Thus $\alpha_1$ is not in the kernel of the top row of the diagram. Hence $\alpha_1$ does \emph{not} represent an element of $\pi^{(m+2)}_r$ under the inclusion.

\end{proof}

\bibliographystyle{plain}
\bibliography{mybib7mathscinet}
\end{document}